\documentclass[a4paper, 11pt]{amsart}

\usepackage[dvipsnames]{xcolor}

\usepackage{iftex}
\ifPDFTeX 
  \usepackage[utf8]{inputenc}
  \usepackage[T1]{fontenc}
\else 
  \usepackage{fontspec}
\fi
\usepackage{lmodern}

\usepackage{geometry}
\usepackage[textwidth=28mm]{todonotes}
\usepackage{mathtools}
\usepackage{yfonts}
\usepackage[foot]{amsaddr}
\usepackage[pdfencoding=auto]{hyperref}
\hypersetup{
    colorlinks,
    citecolor=green,
    filecolor=black,
    linkcolor=blue,
    urlcolor=black
}

\usepackage[english]{babel}

\usepackage{dsfont}
\usepackage{placeins}

\usepackage{caption}
\usepackage{comment}
\usepackage{enumerate}
\usepackage[shortlabels]{enumitem}
\usepackage{tikz}
\usepackage{pgfplots}
\pgfplotsset{compat=1.15}
\usetikzlibrary{arrows}
\usepackage{xcolor}
\definecolor{gemblue}{RGB}{0,114,189}

\numberwithin{equation}{section}

\usepackage{tabularx}

\newcommand{\nocontentsline}[3]{}
\let\origcontentsline\addcontentsline
\newcommand\stoptoc{\let\addcontentsline\nocontentsline}
\newcommand\resumetoc{\let\addcontentsline\origcontentsline}

\theoremstyle{plain}
\newtheorem{thm}{Theorem}[section]

\newtheorem{rem}[thm]{Remark} 

\newtheorem{prop}[thm]{Proposition} 
\theoremstyle{definition} 
 
\theoremstyle{remark}

\newcommand{\R}{\mathbb{R}}

\newcommand{\N}{\mathbb{N}}

\newcommand{\eps}{\varepsilon}

\definecolor{codegreen}{rgb}{0,0.6,0}
\definecolor{codegray}{rgb}{0.5,0.5,0.5}
\definecolor{codepurple}{rgb}{0.58,0,0.82}
\definecolor{backcolour}{rgb}{0.95,0.95,0.92}

\author{Tommaso Tenna}
\address[Tommaso Tenna]{Dipartimento di Matematica, Sapienza Università di Roma, P.le Aldo Moro 5, 00185 Rome, Italy -- Université Côte d’Azur, CNRS, LJAD, Parc Valrose, F-06108 Nice, France}
\email[Tommaso Tenna]{tommaso.tenna@uniroma1.it -- tommaso.tenna@univ-cotedazur.fr}

\date{}

\title[Relaxation Schemes for Flows in Networks]{Relaxation Schemes for Flows in Networks: Application to shallow water and blood flow equations}

\begin{document}

\begin{abstract}
    A numerical scheme of relaxation type is proposed to approximate hyperbolic conservation laws in canal networks. Physical conditions at the junction are given and a novel strategy based on \cite{BNR2025} is introduced to approximate the solution, avoiding the use of approximate Riemann solvers. This general approach is applied to shallow water and blood flow equations, dealing both the subcritical and the supercritical case. The relaxation scheme is complemented with a well-balanced strategy to treat source terms. We investigate properties of the numerical scheme and we present many numerical tests in different settings.\\
    
    \medskip
    
    \textsc{2020 MS Classification: 35R02, 
    35L60, 
    35L65 
    76M12, 
    65M08 
    92C42. 
    }
\end{abstract}

\maketitle

\section{Introduction}
In recent years, the interest for models of conservation laws on networks has significantly increased. Several applications can be found, ranging from vehicular traffic on road networks \cite{garavellopiccoli2006, brettinatalinipiccoli2007, colombogoatinpiccoli2010, borsche2014_traffic} to biological networks describing the movement of bacteria, cells and other microorganisms \cite{borsche2014, brettinataliniribot2014, brettinatalini2018}. For a comprehensive overview of the topic, the reader is referred to \cite{bressanpiccoli2014}. The common features of all these models are the description of the state space through a topological graph (the network) and a dynamics given by solutions to systems of partial differential equations. One of the most widely investigated topic concerns water channel networks \cite{bastin2009open, brianipiccoli2016} and its extension to blood flow arterial networks \cite{canickim2003, muller2015, lucca2025}.\\
The literature on numerical schemes for simulating water flows in single channel is quite rich. Shallow water models are based on hyperbolic systems of conservation laws with possible source terms to describe bottom topography and several numerical methods have been proposed to address difficulties arising in their approximation, satisfying steady-state-preserving and positivity-preserving properties, entropy inequalities and stability properties at discontinuous bottom, see for instance \cite{perthame2001, audusse2004, noelle2006, castro2006, fjordholm2011, dansac2016, ranocha2017} and references therein. The extension of numerical methods to networks has been investigated in different frameworks \cite{brianipiccoli2016, brianipiccoli2018}, but rarely addressing problems concerning the presence of source terms in the dynamics. 
Over the last decades, also the mathematical description of blood flows has been investigated, proposing several models and approaches to simulate the phenomenon, see \cite{formaggiaquarteroni2009, bertagliapareschi2023}. One of the simplest hyperbolic model has been proposed in \cite{canickim2003}, obtained as simplification of the incompressible axisymmetrical Navier-Stokes equations. This particular model, enriched with further hypotheses on the shape of the velocity profile, leads to a shallow water-type model for $1$D blood flows.\\ 
Such hyperbolic systems have two genuinely nonlinear characteristic field and may exhibit different regimes:  \textit{subcritical} (or \textit{fluvial}) for eigenvalues of opposite sign, namely when there are two waves propagating in opposite directions, and \textit{supercritical} (or \textit{torrential}) when both eigenvalues are positive.\\
The main difficulty arising in the numerical approximation of such hyperbolic conservation laws on networks is due to the junction conditions \cite{bressanpiccoli2014}. Several conditions may be naturally imposed to satisfy conservation properties at the junction. The most classical and natural condition is the conservation of fluid mass, namely the assumption that mass is not dispersed at the junction and the sum of fluxes in the incoming branches equal fluxes from the outgoing ones. Additional conditions are usually conservation of energy or pressure. We refer the reader to \cite{colombo2008} for details on coupling conditions in the subcritical regime and \cite{gugat2004} for discussion on supercritical one.\\
These assumptions are usually not sufficient to get a well-posed problem for flows in canal networks and further conditions at junctions are required. A possibility consists in considering Riemann solvers at the junction to guarantee the admissibility of solutions \cite{garavellopiccoli2009}. In \cite{brianipiccoli2016} the authors solve the problem analytically, restricting to the case of subcritical case, necessary to provide a good definition of solution on the network without restrictions. The treatment of supercritical regime is more challenging, due to the absence of unique solutions in all cases. In \cite{brianipiccoli2018}, the authors propose a method based on a detailed analysis of the geometry of Lax curves, looking at the intersection of the admissibility regions in the subcritical set. This strategy allows to recover well-posedness of the problem, but it disregards other possible existing solutions.\\
The goal of this paper is extending the strategy of Briani, Natalini and Ribot \cite{BNR2025}, by proposing a numerical scheme based on relaxation methods to avoid the use of nonlinear Riemann solvers at the junction \cite{jinxin1995, aregba2000discrete}. The idea relies on a Bhatnagar-Gross-Krook (BGK) approximation of the macroscopic equation, where both the source term depend singularly on a relaxation parameter. The main advantage of relaxation schemes is their structure, which is very likely for numerical and theoretical purpose \cite{aregba2000discrete, aregba1996}. The nonlinearity inside the derivatives is replaced by a semilinearity: the differential part becomes linear and all the nonlinearity is concentrated inside the source term. The discrete BGK formulation consists in a linear advection with a relaxation-type source term, which implies that for the homogeneous equation the solution of the Riemann problem at the junction can be explicitly computed, without any restrictions.\\
The presence of source terms in the model may cause problems due to errors arising in the numerical approximation. Several techniques have been proposed to overcome the difficulty of approximating solutions close to steady states. Well-balanced techniques \cite{greenberg1996, gosse2000, gosse2001, perthame2001, gossetoscani2003, audusse2004, noelle2006, muller2013, pupposemplice2016} have been designed to exactly preserve particular steady-state solutions of the system. Indeed, small perturbations around a steady state may cause large spurious numerical errors if the scheme does not catch the equilibrium correctly. In the particular framework of shallow waters and blood flows, we have adapted to the case of networks the strategy proposed in \cite{audusse2004, delestre2013}, based on reconstruction at the interface.

\subsubsection*{Aim of the paper}
In this manuscript we focus on the construction of a numerical scheme for the solution of hyperbolic models in networks, based on discrete-BGK approximation. Inspired by the strategy proposed in \cite{BNR2025}, this approach is used for the approximation of shallow water equations in canal networks and blood flow equations in arterial networks. Subcritical and supercritical regimes can be easily investigated, without the use of approximate Riemann solvers. Differently from other works in the literature, we take into account also source terms, by employing well-balanced techniques \cite{audusse2004, delestre2013} in the case of networks. Numerical properties of the scheme like preservation of the mass, positivity of the density, discrete entropy dissipation and preservation of the steady state, are investigated and validated through numerical simulations.

\subsubsection*{Structure of the paper}
The rest of the paper is organized as follows. In Section \ref{Section:MathModels} we introduce the mathematical models analyzed here, the shallow water equations and the blood flow equations, detailing the properties of the governing systems. In Section \ref{flows_canal_network} we describe the main properties of flows in canal networks, focusing on the junction conditions. Section \ref{Section:Discrete_Kinetic} is devoted to the construction of relaxation schemes on networks, including a theoretical analysis of the main properties. In Section \ref{Section:JRP} we define the junction Riemann problem for different configurations which will be investigated in the numerical simulations. Finally, in Section \ref{Section:Numerical_Simulations} we illustrate the results with some numerical tests, both for the shallow water system and for the blood flow system.

\section{Mathematical Models}
\label{Section:MathModels}
\subsection{The Shallow Water Equations}
The Saint-Venant model for shallow water flows is a hyperbolic system of conservation laws introduced to describe stiff phenomena such as rivers or coastal areas flows, hydraulic jumps and dam breaks. Let us consider a system describing the water propagation in a canal with rectangular cross-section and constant slope. In particular, the conservative variable are given by the water height $h$ and the discharge $h\,v$ whereas the pressure term is assumed to be the hydrostatic pressure. The system in the one-dimensional case reads as 
\begin{equation}
\label{shallow_system}
    \begin{cases}
        \partial_t h + \partial_x (hv) = 0,\\[7pt]
        \partial_t (hv) + \partial_x \left(hv^2+\displaystyle \frac{1}{2}gh^2 \right) = - gh \partial_x z,
    \end{cases}
\end{equation}
where $x$ denotes the location along the canal, $v$ is the water velocity at time t and position $x$, $g$ is the gravity constant, $z$ is the bottom-topography. For the purpose of this work, we neglect more complex effects, like two-dimensional effects, wind forces or Coriolis forces arising in a rotational frame \cite{audusse2021, desveaux2022}. \\
By defining the following quantities
\begin{equation}
    \label{flux_shallow}
    u := \displaystyle \begin{pmatrix}
        h\\[5pt] hv
    \end{pmatrix}, \qquad F(u) := \begin{pmatrix}
    hv\\[5pt] h\,v^2 + \displaystyle\frac{1}{2} g\,h^2
    \end{pmatrix}, \qquad \mathcal{S} (u) := \begin{pmatrix}
    0\\[5pt] -gh\,z_x
    \end{pmatrix} 
\end{equation}
we can rewrite system \eqref{shallow_system} in a compact form as
\begin{equation}
    \partial_t u + \partial_x F(u) = \mathcal{S} (u).
\end{equation}
In this case, the flux Jacobian is given by
\begin{equation}
	A(u)=F'(u)=
	\begin{pmatrix}
		0 & 1\\
		-v^2 + gh & 2v
	\end{pmatrix}.
\end{equation}
Its eigenvalues are
\begin{equation}
	\lambda_1=v-\sqrt{gh}, \qquad \lambda_2=v+\sqrt{gh},
\end{equation}
with corresponding eigenvectors
\begin{equation}
	r_1=\begin{pmatrix}
		1\\
		v-\sqrt{gh}
	\end{pmatrix}, \qquad 
	r_2=\begin{pmatrix}
		1\\
		v+\sqrt{gh}
	\end{pmatrix}.
\end{equation}
In general the eigenvalues $\lambda_i$ can be of either sign and when the velocity is smaller than the speed of the gravity waves, namely
\begin{equation}
\label{eq::chap4_subcritical_condition}
    |v| < \sqrt{gh},
\end{equation}
the fluid is said to be \textit{fluvial} or \textit{subcritical} and then one has 
\begin{equation}
    \lambda_1 (u) \, < \, 0 \, < \, \lambda_2 (u).
\end{equation}
Thus, in the fluvial regime there are two waves propagating in opposite directions.\\
If the velocity is greater than the speed of the gravity waves, namely
\begin{equation}
    |v| \geq \sqrt{g\,h},
\end{equation}
the fluid is said to be \textit{torrential} or \textit{supercritical} and one has
\begin{equation}
    0 \, < \, \lambda_1(u) \, < \, \lambda_2(u).
\end{equation}
In the torrential regime, both waves are propagating in the same direction. Finally, we have the critical case $\lambda_1(u) = \lambda_2(u)=0$.
We define $q=hv$ the \textit{discharge}, which is the conserved quantity of the system. These definitions are usually given in terms of \textit{Froude number}, which is the ratio $Fr = |v| / \sqrt{gh}$.\\

\subsection{The Arterial Network}
\label{Subsection::Blood_Flow}
    Let us consider the hyperbolic model proposed in \cite{canickim2003} to describe blood flow through the arterial vessel in the one-dimensional setting
	\begin{equation}
    \label{System_Arterial}
		\begin{cases}
			\partial_t a + \partial_x (av) = 0,\\
			\partial_t (av) + \partial_x (\alpha a v^2) + \displaystyle \frac{a}{\rho} \partial_x p(a) = -2 \displaystyle \frac{\alpha}{\alpha - 1} \nu v,
		\end{cases}
	\end{equation}
	where $x$ denotes the distance along the center line of the artery, $a$ is the section area of a vessel at position $x$ and time $t$, $v$ is the cross-sectional average fluid velocity, $\rho$ is the constant density, $\nu$ is the kinematic viscosity coefficient and $\alpha$ is a coefficient taking into account the correction due to the conservation of the momentum. In the model $p(a)$ is the fluid pressure function, depending on the cross-sectional area $a$. In particular, we consider the following constitutive relation linking $p$ and $a$:
	\begin{equation}
        \label{pressure_constitutive_arterial}
		p(a) = p_{ref} + \frac{h\,E}{(1-\sigma^2)R_0} \left( \sqrt{\frac{a}{a_0}} -1 \right) = p_{ref} + \beta  \left( \sqrt{\frac{a}{a_0}} -1 \right).
	\end{equation}
    Here, $R_0$ denotes the radius of the reference cylinder approximating the vessel, $a_0=R_0^2$, $h$ is the thickness of the arterial wall, $E$ is the Young's modulus of elasticity and $\sigma$ is the Poisson ratio. The external pressure $p_{ref}$ denotes the pressure at which the displacement of the arterial wall is zero ($a=a_0$). For the sake of simplicity, we will consider $p_{ref}=0$.\\
	Assuming smooth solutions $(a, a\,v)$, we can rewrite the system in conservation form 
	\begin{equation}
    \label{System_Arterial_Conservation}
		\begin{cases}
			\partial_t a + \partial_x q = 0,\\
			\partial_t q + \partial_x \left(\displaystyle \alpha \frac{q^2}{a} + \frac{1}{\rho} \pi(a) \right) = -2 \displaystyle  \frac{\alpha}{\alpha - 1} \nu \frac{q}{a},
		\end{cases}
	\end{equation}
	where now
	\begin{equation}
		q=a\,v \qquad \text{ and } \qquad \pi(a) = \int_{a_0}^a \tilde{a} p'(\tilde{a}) d\tilde{a}.
	\end{equation}
        Using the particular form of pressure we have defined in \eqref{pressure_constitutive_arterial}, we have
        \begin{equation}
        \label{pressure_pi_arterial}
            \pi(a) = \frac{1}{3} \beta \left(\sqrt{\frac{a^3}{a_0}} - a_0 \right).
        \end{equation}
	Thus, we define the following quantities
	\begin{equation}
        \label{flux_arterial}
		u := \displaystyle \begin{pmatrix}
			a\\[5pt] q
		\end{pmatrix}, \qquad F(u) := \begin{pmatrix}
		q\\[5pt] \displaystyle \alpha \frac{q^2}{a} + \frac{1}{\rho} \pi(a)
		\end{pmatrix}, \qquad \mathcal{S}(u) := \begin{pmatrix}
		    0,\\[5pt] -2 \displaystyle  \frac{\alpha}{\alpha - 1} \nu \frac{q}{a}
		\end{pmatrix}
	\end{equation}
    in order to rewrite in compact form the system as
    \begin{equation}
        \partial_t u + \partial_x F(u) = \mathcal{S} (u). 
    \end{equation}
    We remark that systems \eqref{System_Arterial} and \eqref{System_Arterial_Conservation} are not equivalent if the solutions are not \textit{smooth}, namely for solutions involving jumps discontinuities.\\
    A further simplification of system \eqref{System_Arterial_Conservation} leads to
    \begin{equation}
    \label{System_Arterial_SaintVenant}
		\begin{cases}
			\partial_t a + \partial_x (av) = 0,\\
			\partial_t v + \partial_x \left(\displaystyle \alpha \frac{v^2}{2} + \frac{1}{\rho} \pi(a) \right) = -2 \displaystyle  \frac{\alpha}{\alpha - 1} \nu v,
		\end{cases}
	\end{equation}
    where in the particular case of $\alpha=1$ (which requires $\nu=0$), the quantity in the spatial derivative of the equation for $v$ is the \textit{total pressure} rescaled w.r.t. the density $\rho$.\\
    Let us focus on the full system \eqref{System_Arterial_Conservation} in conservation form and let us consider the flux Jacobian matrix $A(u)$:
	\begin{equation}
		A(u) = F'(u) = \begin{pmatrix}
			0 & 1\\  -\displaystyle \alpha \frac{q^2}{a^2} + \frac{1}{\rho} \pi'(a) & \displaystyle \frac{2 \alpha q}{a}
		\end{pmatrix}.
	\end{equation} 
	The eigenvalues of this matrix are given by
	\begin{equation}
		\begin{aligned}
		\lambda_{1,2} = &\alpha \frac{q}{a} \pm \sqrt{\alpha^2 \frac{q^2}{a^2} - \alpha \frac{q^2}{a^2} + \frac{1}{\rho} \pi'(a)} = \\ =&\alpha \frac{q}{a} \pm \sqrt{\alpha(\alpha - 1) \frac{q^2}{a^2} + \frac{1}{\rho} \pi'(a)}.
		\end{aligned}
	\end{equation}
	with corresponding eigenvectors
	\begin{equation}
    \label{eigenvectors_blood}
		r_{1,2}=\begin{pmatrix}
			1\\[8pt]
			\displaystyle \alpha \frac{q}{a} \pm \sqrt{\alpha(\alpha - 1) \frac{q^2}{a^2} + \frac{1}{\rho} \pi'(a)}
		\end{pmatrix}.
	\end{equation}
	The eigenvalues and the eigenvectors depend on $q$, so the waves appearing in the Riemann problem move at different speed and may be jump discontinuities or smoothly varying rarefaction waves. Also for this system, we can identify different regimes, according to the signs of the eigenvalues $\lambda_i$. When the velocity is smaller than the Moens Korteweg wave propagation speed \cite{pedley1980}, namely
    \begin{equation}
    \label{eq::chap4_subcritical_condition}
        |v| < c := \sqrt{\frac{\beta \sqrt{A}}{2\rho\sqrt{\pi}}},
    \end{equation}
    the flow in arteries is said to be \textit{subcritical} and then one has 
    \begin{equation}
        \lambda_1 (u) \, < \, 0 \, < \, \lambda_2 (u).
    \end{equation}
    If the velocity is greater than the Moens Korteweg propagation speed, namely
    \begin{equation}
        |v| \geq c,
    \end{equation}
    the flow in arteries is said to be \textit{supercritical} and one has
    \begin{equation}
        0 \, < \, \lambda_1(u) \, < \, \lambda_2(u).
    \end{equation}
    The supercritical regime is not really a relevant case for flows in arteries and it will not be investigated in the numerical simulations \cite{pedley1980, delestre2013}.
    Finally, as for the shallow water equations, we identify the critical case as $\lambda_1(u) = \lambda_2(u)=0$.

\section{Flows in Canal Networks}
\label{flows_canal_network}
Canal networks are described by topological graphs, namely couples $(\mathcal{I}, \mathcal{J})$, where $\mathcal{I}$ is a collection of intervals representing canals (usually $\mathcal{I} \subset \R$) and $\mathcal{J}$ is a collection of vertices representing junctions (or nodes). To study the interaction of waves described by
the system in canal networks, we need to focus on the behavior at the junction: the solutions holding along each canal must be coupled at the junction by imposing some physically reasonable coupling conditions. For the sake of simplicity, we restrict the discussion to the case of a single junction and we assume that different canals are connected at $x=0$. The typical conditions are:
\begin{itemize}
\item the \textit{conservation of mass} at the node, and

\item the \textit{continuity of flow pressure} (or even total pressure) at the node.
\end{itemize}
Let us suppose that there are $n$ canals meeting at a node, which do not necessarily point away from the node. The continuity
of mass flux at a node for the shallow water system is then equivalent to
\begin{equation}
  \sum_{i: \text{incoming}} (hv)_i = \sum_{j:\text{outgoing}} (hv)_j, \qquad \forall \, t >0.
\end{equation}
The parametrization typically reflects the direction of flow. The conservation of mass can be easily extended to the case in which the canals have different cross sections. Let $\alpha_k$ be the constant cross section of the $k$-th canal $\mathcal{I}_k$, the momentum is then given by $q_k=\alpha_k\, h_k, \, v_k$ and we can rewrite the continuity of mass flux as 
\begin{equation}
\label{junction_mass_cons_shallow}
  \sum_{i: \text{incoming}} (\alpha \,hv)_i = \sum_{j:\text{outgoing}} (\alpha \, hv)_j, \qquad \forall \, t >0.
\end{equation}
For the arterial network, the continuity of mass flux becomes essentially
\begin{equation}
\label{junction_mass_cons_arterial}
  \sum_{i: \text{incoming}} (av)_i = \sum_{j:\text{outgoing}} (av)_j, \qquad \forall \, t >0.
\end{equation}
The continuity of flow pressure (or total pressure) at the node is usually given by the coupling condition which translates in 
\begin{equation}
\label{junction_pressure}
    p_k = p_\ell, \qquad \forall\, t>0,
\end{equation}
which, for the shallow water equations, becomes
\begin{equation}
\label{junction_pressure_shallow}
    \frac{1}{2} g h_k^2 = \frac{1}{2} g h_\ell^2, \qquad \forall\, t>0.
\end{equation}
An alternative choice to the continuity of the flow pressure is given by the energy continuity, namely for the shallow water equations 
\begin{equation}
    h_k+\frac{v_k^2}{2g} = h_\ell+\frac{v_\ell^2}{2g}, \qquad \forall\, t>0.
\end{equation}
Which conditions are used often depends on the regime of the flow (subcritical or supercritical). We suggest referring to \cite{bressanpiccoli2014} therein. In general, these conditions are not sufficient to get a well-posed problem. Most of the literature considers extra conditions by imposing the correct physical wave at the junction, through the solution of the Riemann problems \eqref{Phi_L_Riemann}-\eqref{Phi_R_Riemann} (for shallow waters) and \eqref{Psi_L_Riemann}-\eqref{Psi_R_Riemann} (for arterial blood flows). This choice restricts the analysis to the case of \textit{subcritical} conditions, with very few extensions to the case of \textit{supercritical} flows, only in the case of simple networks see \cite{brianipiccoli2018}. The idea is avoiding the resolution of a Riemann problem at the junction at each iteration. To this aim, following \cite{BNR2025}, we propose a relaxation approximation of the system to recover more conditions on the state variable at the junction and obtain a well-posed problem to solve.\\
From now on, we introduce the following notation, which allows dealing more complex networks. Each canal $i$ is parametrized in $[0,L_i]$, where the junction coincides with $x=L_i$ for the incoming canals and with $x=0$ for the outgoing ones. In particular, when we describe a 2-canal network, the quantities linked to the left incoming canal (resp. the right outgoing canal) are denoted with subscript $\ell$ (resp. $r$).\\
We can rewrite explicitly systems \eqref{shallow_system}-\eqref{System_Arterial} on a 2-canal network in compact form as
\begin{equation}
\label{general_2canal}
    \begin{cases}
        \partial_t u_\ell + \partial_x F(u_\ell) = \mathcal{S} (u_\ell), \qquad \text{for } \,\, x \in [0, L_\ell],\\
        \partial_t u_r + \partial_x F(u_r) = \mathcal{S} (u_r), \qquad \text{for } \,\, x \in [0, L_r].
    \end{cases}
\end{equation}
Finally, each system  has to be complemented by initial condition and appropriate boundary conditions on the external boundary of the network. In particular, for the shallow water system \eqref{general_2canal} with \eqref{flux_shallow}, let us consider no-flux boundary conditions at the external points of the canals, namely for any incoming canal in the interval $[0, L_\ell]$
\begin{equation}
\label{boundary_ingoing}
    \partial_x h_i(t, 0) = 0, \quad v_i(t, 0) = 0, \qquad \text{for all } t \geq 0
\end{equation}
and for any outgoing canal in the interval $[0,L_r]$ 
\begin{equation}
\label{boundary_outgoing}
    \partial_x h_j(t, L_r) = 0, \quad v_j(t, L_r) = 0, \qquad \text{for all } t \geq 0.
\end{equation}
Analogously, for the blood flow system \eqref{general_2canal} with \eqref{flux_arterial}, we complement the system with the same boundary conditions, where the height $h$ is replaced by the section area $a$.

\subsection{Entropy Dissipation}
Let us investigate some properties of the systems \eqref{shallow_system}-\eqref{System_Arterial}, proving an entropy dissipation property at the continuous level at least for smooth solutions.
\subsubsection{Shallow Water Equations}
Let us consider first the shallow water equations \eqref{shallow_system} without bottom topography, which satisfies an entropy dissipation property at the continuous level. Indeed, for each canal, for admissible solutions it holds the following inequality
\begin{equation}
    \partial_t \eta(U) + \partial_x G(U) \leq 0,
\end{equation}
where the entropy-entropy flux pair $(\eta, G)$ are defined as
\begin{equation}
\label{entropy_pair}
    \eta(U) = \frac{1}{2} h v^2 + \frac{g}{2} \, h^2, \qquad G(U) = \left(\frac{1}{2} h\,v^2 + g \, h^2  \right) v.
\end{equation}
This means that on each canal with no flux boundary conditions, the quantity
\begin{equation}
    \int_0^L \eta(U(t,x)\,dx
\end{equation}
is decreasing in time.\\
Let us now consider smooth solutions on a 2-canal network $(\rho_\ell,v_\ell)$ and $(\rho_r,v_r)$, respectively. Then, it holds
\begin{multline}
    \int_0^{L_\ell} \eta(U_\ell (t,x))\,dx + \int_0^{L_r} \eta(U_r(t,x))\,dx \\ \leq \int_0^{L_\ell} \eta(U_\ell (0,x))\,dx + \int_0^{L_r} \eta(U_r(0,x))\,dx \\ - \int_0^t \left[ G(U_\ell (s,L_\ell)) - G(U_r(s,L_r)) \right]\,ds \\ - \int_0^t \left[ G(U_\ell (s,0)) - G(U_r(s,0)) \right]\,ds.
\end{multline}
Using boundary conditions \eqref{boundary_ingoing}-\eqref{boundary_outgoing} and junction conditions \eqref{junction_mass_cons_shallow}-\eqref{junction_pressure_shallow} at the junction node, we easily recover entropy dissipation, namely for all $t>0$ we have
\begin{equation}
\label{entropy_inequality_junction_shallow}
     \int_0^{L_\ell} \eta(U_\ell (t,x))\,dx + \int_0^{L_r} \eta(U_r(t,x))\,dx \leq \int_0^{L_\ell} \eta(U_\ell (0,x))\,dx + \int_0^{L_r} \eta(U_r(0,x))\,dx.
\end{equation}
Indeed, thanks to the boundary conditions on the external domain, it holds
\begin{equation}
    G(U_r(s, L_r)) = G(U_\ell(s,0)) = 0,
\end{equation}
and, using conditions at the junction, the quantity
\begin{multline}
    G(U_\ell(s, L_\ell)) - G(U_r(s,0)) = \frac{1}{2} \left( \frac{(q_\ell(s,L_\ell))^3}{\rho_\ell(s,L_\ell))^2} - \frac{(q_r(s,0))^3}{\rho_r(s,0))^2} \right) \\+ g \left( \rho_\ell(s,L_\ell) q_\ell (s,L_\ell) - \rho_r (s,0) q_r (s,0) \right) = 0.
\end{multline}
In presence of bottom topography, considering the source term $\mathcal{S}$ in \eqref{flux_shallow}, the entropy-entropy flux pair is corrected as \cite{bouchut2004}
\begin{equation}
    \tilde{\eta}(U,z ) = \eta(U) + h g S, \qquad \tilde{G}(U) = G(U) + g\,q\,z,
\end{equation}
which satisfies on each canal
\begin{equation}
    \partial_t \tilde{\eta}(U) + \partial_x \tilde{G}(U) \leq 0,
\end{equation}
For a 2-canal network, we recover entropy dissipation by noticing that Eq. 
\eqref{entropy_inequality_junction_shallow} still holds for $\left( \tilde{\eta}, \tilde{G} \right)$, since the bottom-topography $z$ is continuous at the junction.

\subsubsection{Blood Flow Equations}
Let us now consider the blood flow equations \eqref{shallow_system}, which satisfies an entropy dissipation property at the continuous level for the case $\nu = 0$. Indeed, for each canal, for admissible solutions it holds the following inequality
\begin{equation}
    \partial_t \eta(U) + \partial_x G(U) \leq 0,
\end{equation}
where for the case $\nu=0$ with pressure \eqref{pressure_pi_arterial}, the entropy-entropy flux pair $(\eta, G)$ are defined 
\begin{equation}
    \eta(U) = \frac{1}{2} a v^2 + \frac{\kappa}{\gamma-1}\,a^\gamma, \qquad G(U) = \left(\frac{1}{2} a\,v^2 + \frac{\gamma\,\kappa}{\gamma-1}\,a^\gamma - a_0^ \gamma \kappa \right) v,
\end{equation}
where $\gamma= \displaystyle\frac{3}{2}$ and $\kappa = \displaystyle\frac{1}{3} \frac{\beta}{\sqrt{a_0}}$.\\
If we consider again smooth solutions on a 2-canal network $(a_\ell, v_\ell)$ and $(a_r, v_r)$ respectively, then the entropy dissipation is recover analogously to the shallow water case above.\\
In the case $\nu \neq 0$ the equation for the flux pair in the arterial network also includes an entropy production due to viscous dissipation
\begin{equation}
    \sigma(U) = 2 \displaystyle \frac{\alpha}{\alpha - 1} \nu v^2.
\end{equation}
This leads to the corrected entropy inequality
\begin{equation}
    \partial_t \eta(U) + \partial_x G(U) \leq \sigma(U)\,v,
\end{equation}
as already observed in \cite{lizhao2011}.

\section{Relaxation schemes on Networks}
\label{Section:Discrete_Kinetic}
Let us introduce a two-velocities relaxation scheme for the approximation of general flows in canal networks. For the sake of completeness, we first recall some generalities about discrete kinetic scheme for a one-dimensional scalar conservation law. This approach is then adapted to the case of a $2$-canal network for the shallow water system \eqref{shallow_system} and the blood flow equations \eqref{System_Arterial_SaintVenant} with proper junction conditions. Extensions to more complex networks are detailed in Section.

\subsection{Discrete kinetic schemes}
Let us consider a general Cauchy problem  for a one-dimensional quasi-linear conservation law without source term
\begin{equation}
\label{Cons_Law_Scalar}
    \begin{cases}
        \partial_t u + \partial_x A(u) = 0,\\
        u(x,0)=u_0(x).
    \end{cases}
\end{equation}
Let us consider a class of numerical schemes based on a discrete kinetic approximation \cite{aregba2000discrete}, which enables to approximate \eqref{Cons_Law_Scalar} by a sequence of semi-linear hyperbolic systems, through a BGK relaxation equation, leading to
\begin{equation}
    \label{semi-linear general}
    \partial_t f^\eps + \Lambda \partial_{x} f^\eps=\frac{1}{\eps} \Big(\mathcal{M}(u^\eps)-f^\eps\Big),
\end{equation}
where the $u^\eps$ variable is given by
\begin{equation}
	\label{u_eps_BGK}
	u^\eps:=\displaystyle \sum_{j=1}^L f_j^\eps.
\end{equation}
The parameter $\eps$ is a positive number, $\Lambda$ is a real diagonal matrix $L\times L$, where $L$ is the number of discrete velocities, $f^\eps=(f_1^\eps,\dots,f_L^\eps)$ and $\mathcal{M}: \R^k \to \R^L$ is a Lipschitz continuous function.\\
The numerical scheme presented in \cite{aregba2000discrete} is constructed by splitting \eqref{semi-linear general} into a homogeneous linear part and an ordinary differential system, exactly solved thanks to the particular structure of the relaxation term.\\
In the following, we restrict the analysis to the case of $L=2$ discrete velocities. The system \eqref{semi-linear general} can be then rewritten as
\begin{equation}
\label{Relaxation_System}
    \begin{cases}
        \partial_t f_1 + \lambda_1\,\partial_x f_1 = \displaystyle \frac{1}{\eps} \left(\mathcal{M}_1 - f_1 \right),\\[10pt]
        \partial_t f_2 + \lambda_2\,\partial_x f_2 =\displaystyle \frac{1}{\eps} \left(\mathcal{M}_2 - f_2 \right),
    \end{cases}
\end{equation}
where $f_j = f_j(x,t) \in \R^2$ and $\lambda_1$, $\lambda_2$ are two parameters to be chosen. Let us introduce Lipschitz continuous Maxwellian functions $\mathcal{M}_i :\R \to \R^2$, which embed the macroscopic variable into the kinetic one. The convergence of the kinetic model to the macroscopic one needs some further assumptions, in particular for any $u \in \R^2$ we must impose the compatibility conditions
\begin{equation}
\label{Compatibility_Conditions}
    \begin{cases}
        \mathcal{M}_1(u) + \mathcal{M}_2(u) = u,\\
        \lambda_1 \, \mathcal{M}_1(u) + \lambda_2 \, \mathcal{M}_2(u) = A(u).
    \end{cases}
\end{equation}
The first property \eqref{Compatibility_Conditions}$_1$ tells us that if we consider a macroscopic variable and we embed it in the kinetic space through $\mathcal{M}_i$, when we project it back, we obtain the original state thanks to \eqref{u_eps_BGK}. The second property \eqref{Compatibility_Conditions}$_2$ is necessary to guarantee the original macroscopic fluxes.\\ 
The compatibility conditions give explicit expressions for the Maxwellian functions
\begin{equation}
    \mathcal{M}_1 = \frac{\lambda_2\,u - F(u)}{\lambda_2-\lambda_1}, \qquad \mathcal{M}_2 = \frac{\lambda_1\,u - F(u)}{\lambda_1-\lambda_2}. 
\end{equation}
Let us denote by $\Delta x$ the spatial discretization step and by $\mathcal{C}_i$ the $j$-th cell of the finite volume discretization
\begin{equation}
    \mathcal{C}_i = [x_{i-1/2}, x_{i+1/2}] \qquad 1 \leq i \leq N,
\end{equation}
where $x_i = \left(i - \frac{1}{2} \right) \Delta x$ and $N$ is the number of cells used to discretize the interval. Furthermore, let $\Delta t$ be the time step and $t_n = n\Delta t$, $n \in \N$ the discrete times. We finally denote by $ f_\Delta^n \approx f(x_i,t_n)$ and $ U_\Delta^n \approx U(x_i,t_n)$ suitable approximations of $f$ and $U$, respectively, on the discretization grid $\Delta = \{x_i\}$. \\
For a given $f_\Delta^{\eps,n}$, the function $f_\Delta^{\eps, n+\frac{1}{2}}$ is an approximate solution at time $t_{n+1}$ of the problem 
\begin{equation}
\label{Homogeneous_Kinetic}
    \begin{cases}
    \partial_t f_1 + \lambda_1 \partial_x f_1 = 0,\\[10pt]
    \partial_t f_2 + \lambda_2 \partial_x f_2 = 0,\\[10pt]
    f(t_n) = f_\Delta^{\eps, n},
    \end{cases}
\end{equation}
The numerical scheme on the linear part will be denoted by $H_\Delta$, such that
\begin{equation}
    f_\Delta^{\eps, n+\frac{1}{2}} = H_\Delta (\Delta t)\, f_\Delta^{\eps, n}.
\end{equation}
Since relaxation system \eqref{Relaxation_System} consists of two coupled linear advection equations, an exact Riemann solver could be provided at the interface of two cells
\begin{equation}
\label{exact_Riemann_solver}
    \mathcal{R}(x/t, f^-, f^+) = \begin{cases}
        (f_1^-, f_2^-) \qquad \text{if } \quad &x/t < \lambda_1,\\
        (f_1^+, f_2^-) \qquad \text{if } \quad  \lambda_1<&x/t < \lambda_2,\\
        (f_1^+, f_2^+) \qquad \text{if } \quad \lambda_2<&x/t,\\
    \end{cases}
\end{equation}
where $f_j^\pm \approx \mathcal{M}_j(u^\pm)$. According to Bouchut \cite{bouchut2004}, $\lambda_1$ and $\lambda_2$ have to be chosen as the eigenvalues of the flux of the original system \eqref{Cons_Law_Scalar}, in order to recover entropy dissipation relations and stability of the method.\\
The contribution of the singular perturbation term is taken into account by solving on $[t_n, t_{n+1}]$ the ordinary differential system 
\begin{equation}
    F' = \frac{1}{\eps} \left( \mathcal{M} (u) - F \right), \qquad u = \sum_{j=1}^L F_j.
\end{equation}
Taking $\eps=0$, the final system for $L=2$ discrete velocities reads as follows
\begin{equation}
\label{Relaxation_Final}
    \begin{cases}
    f_\Delta^{n+\frac{1}{2}} = H_\Delta (\Delta t) f_\Delta^{n}\\[10pt]
    U^{n+1} = f_1^{n+\frac{1}{2}}+f_2^{n+\frac{1}{2}},\\[10pt]
    f_j^{n+1} = \mathcal{M}_j \left(U^{n+1}\right), \qquad j=1,\,2.
    \end{cases}
\end{equation}
Rewriting the scheme in terms of the macroscopic variable $U^n_\Delta$, we have
\begin{equation}
    U_i^{n+1} = U_i^n - \frac{\Delta t}{\Delta x} \left( \mathcal{F}^n_{i+1/2} - \mathcal{F}^n_{i-1/2}\right),
\end{equation}
where $\mathcal{F}^n_{i+1/2} = \mathcal{F}(U^n_i, U^n_{i+1})$. The conservative flux is equivalent by construction to \cite{bouchut2004, BNR2025}
\begin{equation*}
\label{flux_HHL_bouchut}
    \mathcal{F}(U^-,U^+) = \begin{cases}
        F(U^-), \qquad &\text{if } \, 0 < \lambda_1,\\
        \displaystyle \frac{\lambda_2 F(U^+) - \lambda_1 F(U^-)}{\lambda_2-\lambda_1}+ \frac{\lambda_1\,\lambda_2}{\lambda_2-\lambda_1} (U^+-U^-) \qquad &\text{if } \, \lambda_1 < 0 < \lambda_2\\
        F(U^-), \qquad &\text{if } \, \lambda_2 < 0,\\
    \end{cases}
\end{equation*}
where $F$ is the continuous flux defined in \eqref{flux_shallow}-\eqref{flux_arterial}.\\
The main advantage of this approximation is the possibility of avoiding the resolution of local Riemann problems in the design of numerical schemes. The particular formulation of the homogeneous system \eqref{Homogeneous_Kinetic} is suitable for networks in the case $L=2$, as shown in the following subsection.

\begin{rem}
    The focus of this work is the numerical treatment of the junction with relaxation schemes. To this aim, the source term is treated by directly projecting the kinetic distribution function $f_j$ onto the Maxwellian space at each time step. Higher-order approximations of discrete kinetic schemes could be obtained by considering the source term with very small $\eps$, for instance using an IMEX approach \cite{pareschirusso2005, boscarino2025}, DeC-IMEX \cite{torloabgrall2020} or projective integration method \cite{lafittemelis2017, tenna2025}.
\end{rem}

\subsection{Discrete kinetic schemes on a 2-canal network}
Let us focus on the case of a simple network with two canals, see Fig. \ref{fig:2canalnetwork}. We denote again $\Delta x$ the spatial discretization step and $\Delta t$ the time discretization step. Let us introduce the following notation for the left canal (respectively, right canal)
\begin{equation}
     \mathcal{C}_{i, \ell} = [x_{i-1/2, \ell}, x_{i+1/2, \ell}], \quad 1 \leq i \leq N_{\ell} \qquad \left( \text{respectively } \mathcal{C}_{i,r} = [x_{i-1/2, r}, x_{i+1/2, r}], \quad 1 \leq i \leq N_{r} \right),
\end{equation}
where $N_\cdot$ is the number of the spatial discretization steps and $x_{i,\cdot} = \left( j - \frac{1}{2} \right) \Delta x$ are the centers of the cells, for $\cdot = \ell, r$ respectively.\\
On each canal we solve the original system by considering the corresponding relaxation system \eqref{Relaxation_System}, where we denote $\lambda_\ell$ and $\lambda_r$ the discrete velocities for the left canal and the right canal, respectively. The numerical approximation is complemented with boundary conditions \eqref{boundary_ingoing}-\eqref{boundary_outgoing} at the 
boundary cells of the network $\mathcal{C}_{0,\ell}$ and $\mathcal{C}_{N_r,r}$, namely
\begin{equation}
\label{boundary_ingoing_discrete}
    h_{0,\ell} = h_{1, \ell} \quad \text{ and } \quad q_{0,\ell} = -q_{1, \ell}
\end{equation}
\begin{equation}
\label{boundary_outgoing_discrete}
    h_{N_r,r} = h_{N_r - 1, r} \quad \text{ and } \quad q_{N_r,r} = -q_{N_r-1, r}
\end{equation}

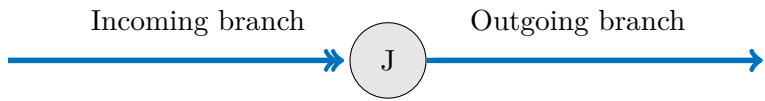
\begin{figure}[ht]
\centering
\begin{tikzpicture}[
    canal/.style={line width=2pt, gemblue},
    node/.style={circle, draw=black, fill=gray!20, minimum size=1cm},
    ->,
    shorten >=2pt
  ]

  \node[node] (junction) at (0,0) {J};
  \coordinate (in) at (-5,0);
  \coordinate (out) at (5,0);

  \draw[canal] (in) -- (junction);
  \draw[canal] (junction) -- (out);

  \draw[canal,->] (-5,0) -- (-0.6,0); 
  \draw[canal,->] (0.6,0) -- (5,0);   

  \node[above] at (-2.5,0.2) {Incoming branch};
  \node[above] at (2.5,0.2) {Outgoing branch};

\end{tikzpicture}
\caption{Simple 2-canal network with one incoming and one outgoing branch, where $J$ denotes the junction.}
\label{fig:2canalnetwork}
\end{figure}

\subsubsection{Numerical treatment of the junction}
\label{Treatment_Junction}
As pointed out in Section \ref{flows_canal_network}, conditions on mass conservation and pressure continuity are not sufficient to get a well-posed problem, since we have $4$ unknowns and $2$ equations. Indeed, at the junction, the mass conservation property at the discrete level becomes
\begin{equation}
\label{junction_mass_cons_discrete}
    h_\ell^*\,v_\ell^* = h_r^*\,v_r^*,
\end{equation}
and the continuity of the flow pressure at the node is given by the coupling condition \cite{brianipiccoli2016, olufsen2000numerical}
\begin{equation}
\label{junction_pressure_discrete}
    p_\ell^* = p_r^*. 
\end{equation}
In the shallow water equations, we recall that condition \eqref{junction_pressure_discrete} implies $h_\ell^* = h_r^*$. In the blood flow equations, $p$ refers to the fluid pressure function in \eqref{pressure_constitutive_arterial}.\\
Inspired by \cite{BNR2025}, we observe that the kinetic approximation gives implicitly supplementary conditions to be satisfied, in order to find a physically correct solution. Indeed, let us suppose for the sake of simplicity that $\lambda_1 < 0 <\lambda_2$, by imposing $\lambda = \lambda_1 = -\lambda_2$, where $\lambda > 0$ could be chosen as $\lambda = \max \left\{ \displaystyle \partial F(u)/\partial u \right\}$. Therefore, by considering the exact Riemann solver \eqref{exact_Riemann_solver} in $x/t = 0$, we obtain that $f_2$ on the the left canal and $f_1$ on the the right canal must be preserved. In terms of the unknown state $(f_1^*, f_2^*)_{\{r,\,\ell\}}$, this results in the following two extra conditions
\begin{equation}
    f_{2,\, \ell}^* = f_{2,\, \ell}^-,
\end{equation}
\begin{equation}
    f_{1,\, \ell}^* = f_{1,\, r}^+.
\end{equation}
This condition is essential to have admissible solutions to the exact Riemann problem at the kinetic level. In other words, this implies that the solution at the junction is determined such that for each $r$ and $\ell$, the intermediate states are connected to the initial condition through physically correct waves. In terms of the macroscopic quantities, by projecting back the kinetic variable through the Maxwellian function, we have
\begin{equation}
    \mathcal{M}_{2}(u_{\ell}^*) = \mathcal{M}_{2}(u_{\ell}^-),
\end{equation}
\begin{equation}
    \mathcal{M}_{1}(u_{r}^*)=\mathcal{M}_{1}(u_{r}^+).
\end{equation}
Let us first focus on the shallow water system \eqref{shallow_system}. According to the choice of the components of the Maxwellian to use in order to introduce the supplementary conditions, we obtain two different systems
\begin{equation}
    \begin{cases}
    \displaystyle \frac{1}{2} \left( h_\ell^* + \displaystyle \frac{q_\ell^*}{\lambda_\ell} \right) = \frac{1}{2} \left( h_\ell^- + \displaystyle \frac{q_\ell^-}{\lambda_\ell} \right),\\[20pt]
    \displaystyle \frac{1}{2} \left( h_r^* + \displaystyle \frac{q_r^*}{\lambda_r} \right) = \frac{1}{2} \left( h_r^+ + \displaystyle \frac{q_r^+}{\lambda_r} \right)
    \end{cases}
\end{equation}
or
\begin{equation}
    \begin{cases}
    \displaystyle \frac{1}{2} \left( q_\ell^* + \displaystyle \frac{ \displaystyle \frac{q_\ell^*}{h_\ell^*} + \displaystyle \frac{1}{2}\,\,(h_\ell^*)^2}{\lambda_\ell} \right) = \displaystyle \frac{1}{2} \left( q_\ell^- + \frac{\displaystyle \frac{q_\ell^-}{h_r^-} + \frac{1}{2}\,\,(h_\ell^-)^2}{\lambda_\ell} \right),\\[28pt]
    \displaystyle \frac{1}{2} \left( q_r^* + \displaystyle \frac{ \displaystyle \frac{q_r^*}{h_r^*} + \displaystyle \frac{1}{2}\,\,(h_r^*)^2}{\lambda_r} \right) = \displaystyle \frac{1}{2} \left( q_r^+ + \frac{\displaystyle \frac{q_r^+}{h_r^+} + \displaystyle \frac{1}{2}\,\,(h_r^+)^2}{\lambda_r} \right)
    \end{cases}
\end{equation}
where $(h_\ell^*, q_\ell^*)$ and $(h_r^*, q_r^*)$, represent the values of the ghost cells at the junction to be used in the numerical scheme.\\
Let us exploit the explicit expression of the system to be solved. For notational convenience, we introduce the following $\Phi_\ell$ and $\Phi_r$ functions
\begin{equation}
\label{Phi_functions}
    \begin{cases}
    \Phi_\ell (h_\ell^*, q_\ell^*) = h_\ell^* + \displaystyle \frac{q_\ell^*}{\lambda_\ell} -  h_\ell^- - \displaystyle \frac{q_\ell^-}{\lambda_\ell},\\
    \Phi_r (h_r^*, q_r^*) = h_r^* + \displaystyle \frac{q_r^*}{\lambda_r} -  h_r^- - \displaystyle \frac{q_r^-}{\lambda_r}.
   \end{cases} 
\end{equation}

\subsection{Well-balanced schemes}
\label{Section:Well_Balanced}
One of the most problematic aspects of classical finite volume schemes is the approximation of solutions close to steady states, since the numerical scheme is usually unable to preserve the hydrostatic balance. More generally, the structure of numerical truncation errors is not compatible with the physical steady state conditions, leading to spurious oscillations. To this aim, in \cite{audusse2004, bouchut2004}, authors have introduced a general strategy to derive a well-balanced scheme, preserving steady states and other fundamental properties of the system.\\
Let us detail the well-balanced scheme both for shallow water and for blood flow equations.\\
This reconstruction procedure is performed at each time step $t^n$ after solving the system to determine values a the junction, before applying the discrete kinetic numerical scheme \eqref{Relaxation_Final}.

\subsubsection{Shallow Water Equations}
Let us recall that a steady state for the shallow water equations \eqref{shallow_system} is given by
\begin{equation}
	hv=\text{cst}, \quad \frac{1}{2}v^2+gw=\text{cst},
\end{equation}
where $w := h+z$. A particular equilibrium state is given by
\begin{equation}
	u \equiv 0 \quad \text{and} \quad w \equiv \text{cst},
\end{equation}
usually known as \textit{lake at rest}. In this case, the water is still, namely $u=0$, and the water height is constant in time, $h(x,t)=h(x)$. Using this assumption in \eqref{shallow_system}, we obtain
\begin{equation}
	\label{eq:HBE}
	\Big(\frac{1}{2}gh^2\Big)_x+ghz_x=0,
\end{equation}
usually called \textit{hydrostatic balance}. The first term is the hydrostatic pressure, describing the behavior of a column of water under the effect of gravity. The second term is the gravitational acceleration due to an inclined bottom topography $z$.\\ 
A finite volume scheme is \textit{well-balanced} if the steady state of a lake at rest is preserved. The difficulty is to get schemes satisfying also conservation properties, non-negativity of the water height $h$, able to compute dry states $h=0$ and satisfying a discrete entropy inequality.\\
Let us describe the strategy to preserve the hydrostatic balance, by considering a conservative finite volume discretization of the hydrostatic pressure in the cell $\Big[x_{i-\frac{1}{2}},x_{i+\frac{1}{2}} \Big]$ for a kinetic scheme. Let us rewrite the kinetic relaxation system \eqref{Relaxation_System} with the Maxwellian compatibility conditions \eqref{Compatibility_Conditions} as
\begin{equation}
    \label{Relaxation_WellBalanced}
    U_i^{n+1} = U_i^n - \frac{\Delta t}{\Delta x} \left( \mathcal{F}^n_{i+1/2} - \mathcal{F}^n_{i-1/2}\right),
\end{equation}
with the numerical flux $\mathcal{F}^n_{i+1/2}$ defined in \eqref{flux_HHL_bouchut}.\\
For a hydrostatic balance, let us suppose the solution $u^n$ is at \textit{lake at rest}, namely
\begin{equation}
    h^n+z^n = h^0+z^0, \qquad v^n = 0. 
\end{equation}
This implies that 
\begin{equation}
    U_i^{n+1} = \begin{pmatrix}
        h_i^n,\\
        \displaystyle \frac{\Delta t}{\Delta x} \left[\displaystyle \frac{g}{2} \left(h^n_{i+\frac{1}{2}^-} \right)^2 - \displaystyle \frac{g}{2} \left(h^n_{i-\frac{1}{2}^+}\right)^2 \right] 
    \end{pmatrix} + \Delta t \, S_i^n
\end{equation}
In addition, let us suppose to discretize the source term as
\begin{equation}
	ghz_x \approx g \bar{h}\,Dz,
\end{equation}
where $\bar{h} \approx h$ and $Dz \approx z_x$. To enforce the hydrostatic balance at time $t^{n+1}$ we set
\begin{equation}
\label{Source_Term_WB}
	S^n_i=\begin{pmatrix}
		0\\
		\displaystyle\frac{g}{2} \left(h^n_{i+\frac{1}{2}^-}\right)^2-\displaystyle \frac{g}{2} \left(h^n_{i-\frac{1}{2}^+}\right)^2 
        \end{pmatrix},
\end{equation}
where we have represented the cell-averaged source term as the discrete gradient of the hydrostatic momentum flux. This ansatz is motivated by a balancing requirement in a nearly hydrostatic regime. Indeed, when $u \ll \sqrt{gh}$, the leading order water height adjusts in order to satisfy the balance of momentum flux and momentum source terms.\\ 
This discretization of the source term is the essential ingredient for a well-balanced formulation, if complemented with some continuity property holding at the equilibrium state.
Indeed, any hydrostatic state is kept exactly if, for such a state, the locally reconstructed heights satisfy
\begin{equation}
	\mathcal{F}_{i+\frac{1}{2}}^{hu}=\frac{1}{2}gh_{i+\frac{1}{2}^-}^2=\frac{1}{2}gh_{i+\frac{1}{2}^+}^2.
\end{equation}
The get a well-balanced scheme, the following hydrostatic reconstruction is introduced \cite{audusse2004, noelle2006}
\begin{equation}
\label{condition_positivity_wellbalanced_shallow1}
    h_{i+\frac{1}{2}^-} = \max \left(0, h_i + z_i - z_{i+\frac{1}{2}} \right), \qquad h_{i+\frac{1}{2}^+} = \max \left(0, h_{i+1} + z_{i+1} - z_{i+\frac{1}{2}}\right),
\end{equation}
\begin{equation}
\label{condition_positivity_wellbalanced_shallow2}
    z_{i+\frac{1}{2}} = \max (z_i, z_{i+1}),
\end{equation}
which is able to guarantee the nonnegativity of the water height.

\subsubsection{Blood Flow Equations}
Let us now consider the blood flow equations \eqref{System_Arterial_SaintVenant}. A steady state for this system is given by
\begin{equation}
	av=\text{cst}, \quad \frac{1}{2} \alpha \frac{q^2}{a} + \frac{\beta}{3\,\rho} \left( \sqrt{\frac{a^3}{a_0}} - a_0 \right)=\text{cst},
\end{equation}
Also in this case we can define a particular equilibrium state given by
\begin{equation}
	u \equiv 0 \quad \text{and} \quad w \equiv \text{cst},
\end{equation}
analogously known as \textit{man at eternal rest} \cite{delestre2013}. In this case, the blood is still, namely $u=0$, and it holds
\begin{equation}
    \partial_x \left(\frac{\beta}{3\rho} a^{3/2} \right) - \frac{\beta\, a}{\rho} \partial_x \sqrt{a_0} = 0, 
\end{equation}
having in this case $\sqrt{a} = \sqrt{a_0}$.\\
Using this assumption in \eqref{System_Arterial_SaintVenant} to construct a well-balanced scheme, we obtain
\begin{equation}
	\label{eq:Blood_Condition_WB}
	\sqrt{a}-\sqrt{a_0}=\text{cst},
\end{equation}
which will be used to reconstruct the variable $a$. 
Let us consider again a conservative finite volume discretization of the hydrostatic pressure in the cell $\Big[x_{i-\frac{1}{2}},x_{i+\frac{1}{2}} \Big]$ for a kinetic scheme. For the kinetic relaxation system \eqref{Relaxation_System} with the Maxwellian compatibility conditions \eqref{Compatibility_Conditions}, we obtain again \eqref{Relaxation_WellBalanced}. Now, the reconstructed values are given
\begin{equation}
\label{condition_positivity_wellbalanced_arterial}
    \begin{cases}
        \sqrt{a_{i+\frac{1}{2}^-}} = \max \left( \sqrt{a_i} + \min \left( \sqrt{a_{0, i+1}} - \sqrt{a_{0, i}}, 0 \right), 0 \right),\\
        \sqrt{a_{i+\frac{1}{2}^+}} = \max \left( \sqrt{a_{i+1}} - \max \left( \sqrt{a_{0, i+1}} - \sqrt{a_{0, i}}, 0 \right), 0 \right).
    \end{cases}
\end{equation}
Finally, the source term could be rewritten as
\begin{equation}
\label{Source_Term_WB_Arterial}
	S^n_i=\begin{pmatrix}
		0\\
		\displaystyle\frac{\beta}{3\rho} \left(a^n_{i+\frac{1}{2}^-}\right)^{3/2}-\displaystyle \frac{\beta}{3\rho} \left(a^n_{i-\frac{1}{2}^+}\right)^{3/2} 
        \end{pmatrix},
\end{equation}
The friction term appearing in \eqref{System_Arterial_SaintVenant}, could be treated implicitly to avoid numerical instabilities.

\subsection{Properties of the scheme}
Let us prove the mass-preserving property, the positivity-preserving property of the density, discrete entropy inequality and the well-balancing property of the scheme in presence of a non-conservative term, as done in \cite{bouchut2004} in absence of junction conditions.
\begin{prop}
    Let us consider the numerical scheme \eqref{Relaxation_Final} on a 2-canal network, with junction condition given by \eqref{junction_mass_cons_shallow}-\eqref{junction_pressure_shallow} and boundary conditions \eqref{boundary_ingoing_discrete}-\eqref{boundary_outgoing_discrete}. The scheme is mass-preserving, namely
    \begin{equation}
        \Delta x \sum_{i=1}^{N_\ell} h_{i,\ell}^n + \Delta x \sum_{i=1}^{N_r} h_{i,r}^n = \Delta x \sum_{i=1}^{N_\ell} h_{i,\ell}^0 + \Delta x \sum_{i=1}^{N_r} h_{i,r}^0.
    \end{equation}
    Moreover, if the kinetic velocities are chosen such that
    \begin{equation}
        \lambda_j^n \geq \max_{1 \leq i \leq N_j} |U_{i,j}^n| + \sqrt{p'} \qquad j \in \{ \ell, r\},
    \end{equation}
    and the initial height (or section $a_i$) is positive, then the solution remains positive in time, namely
    \begin{equation}
        h^0_{i,j} \geq 0 \quad \forall 1 \leq i \leq N_j \quad \implies \quad h^n_{i,j} \geq 0 \quad \forall \, 1 \leq i \leq N_j, \quad \forall \, n \in \N,  \qquad j \in \{ \ell, r\}.
    \end{equation}
\end{prop}
\begin{proof}
    It is straightforward to observe that the mass is preserved at the discrete level, due to the boundary conditions \eqref{boundary_ingoing_discrete}-\eqref{boundary_outgoing_discrete} and the junction condition \eqref{junction_mass_cons_discrete}. Moreover, condition \eqref{junction_pressure_shallow}  guarantees the continuity of the height (or the vessel section) at the interface, which implies that if $h_i^0 \geq 0$ for all $i$, then the solution remains positive in time. Indeed, the truncation \eqref{condition_positivity_wellbalanced_shallow1}-\eqref{condition_positivity_wellbalanced_shallow2} (or equivalently \eqref{condition_positivity_wellbalanced_arterial} for the arterial network), implies that the difference $\mathcal{F}_{i+1/2}^n-\mathcal{F}_{i-1/2}^n \leq 0$, whenever $h_i=0$ (or $a_i=0$).
\end{proof}

Let us now focus on the entropy dissipation properties of the numerical scheme. Differently from \cite{BNR2025}, junctions conditions automatically ensure entropy dissipation without further assumptions.
\begin{prop}
    Let us consider the numerical scheme \eqref{Relaxation_Final} on a 2-canal network for the shallow water system \eqref{shallow_system} without source term, with junction condition given by \eqref{junction_mass_cons_shallow}-\eqref{junction_pressure_shallow} and boundary conditions \eqref{boundary_ingoing_discrete}-\eqref{boundary_outgoing_discrete}. There exists a numerical entropy flux function $\mathcal{G}$, which satisfies the following inequality related to the discrete entropy $\eta$
    \begin{equation}
        \sum_{j \in \{ \ell, r\}} \sum_{i=1}^{N_j} \left[\eta ( U^{n+1}_i) - \eta ( U^n_i) + \frac{\Delta t}{\Delta x} \left( \mathcal{G}(U^n_i, U^n_{i+1}) - \mathcal{G}(U^n_{i-1}, U^n_i)\right) \right] \leq 0.
    \end{equation}
    \begin{proof}
        The critical point is the junction, since under CFL condition, on each arc it holds
        \begin{equation}
            \eta ( U^{n+1}_i) - \eta ( U^n_j) + \frac{\Delta t}{\Delta x} \left( \mathcal{G}(U^n_j, U^n_{i+1}) - \mathcal{G}(U^n_{i-1}, U^n_i)\right) \leq 0 \qquad \forall \, i,
        \end{equation}
        as shown in \cite{audusse2004}, where $\mathcal{G}$ is the discrete counterpart of the entropy flux defined in \eqref{entropy_pair}. By summing over all the cells and over all the arcs, the inequality on the whole domain becomes
        \begin{multline}
        \left(\sum_{i=1}^{N_\ell} \eta ( U^{n+1}_{i,\ell}) + \sum_{i=1}^{N_r} \eta ( U^{n+1}_{i,r}) \right) - \left(\sum_{i=1}^{N_\ell} \eta ( U^n_{i,\ell}) + \sum_{i=1}^{N_r} \eta ( U^n_{i,r}) \right) \\ +\frac{\Delta t}{\Delta x} \left( -\mathcal{G}^n_{1/2,\ell} + \mathcal{G}^n_{N_\ell+ 1/2,\ell} -\mathcal{G}^n_{1/2,r} +  \mathcal{G}^n_{N_r+1/2,r} \right) \leq 0.
    \end{multline}
    Indeed, from boundary conditions \eqref{boundary_ingoing_discrete}-\eqref{boundary_outgoing_discrete}, for the shallow water equations, we have
    \begin{equation}
        \mathcal{G}^n_{1/2,\ell} + \mathcal{G}^n_{N_r+ 1/2,r} = 0,
    \end{equation}
    and from junction conditions we have also
    \begin{equation}
        \mathcal{G}^n_{N_\ell+ 1/2,\ell} + \mathcal{G}^n_{1/2,r} = 0.
    \end{equation}
    \end{proof}
\end{prop}

Extensions of discrete entropy dissipation in presence of source terms can be derived in the same spirit of \cite{bouchut2004, audusse2004}, but it goes beyond the scope of this work.
Let us finally underline the well-balanced property of our scheme, as stated.
\begin{prop}
    Let us consider the numerical scheme \eqref{Relaxation_Final} on a 2-canal network, with junction conditions given by \eqref{junction_mass_cons_shallow}-\eqref{junction_pressure_shallow} and boundary conditions \eqref{boundary_ingoing_discrete}-\eqref{boundary_outgoing_discrete}. Then, the scheme is well-balanced, namely preserves the steady state of a lake at rest.
\end{prop}
\begin{proof}
    To prove the well balanced property, let us consider the steady state of a lake at rest, with $h_{i+\frac{1}{2}^-} = h_{i+\frac{1}{2}^+}$ and $v_j = v_{j+1} = 0$, from which we have $U_{i+\frac{1}{2}^-} = U_{i+\frac{1}{2}^+}$.\\
    By consistency of $\mathcal{F}$ we have $F_{i+\frac{1}{2}} = F(U_{i+\frac{1}{2}^+}) = F(U_{i+\frac{1}{2}^-})$ which, together with the expression of the source terms \eqref{Source_Term_WB} gives the well balanced property in the internal cells. Conditions \eqref{junction_mass_cons_discrete}-\eqref{junction_pressure_discrete} guarantee $F_{i+\frac{1}{2}} = F \left(U_{i+\frac{1}{2}^+} \right) = F \left(U_{i+\frac{1}{2}^-} \right)$ also at the junction, concluding the proof.
\end{proof}

\section{The Junction Riemann Problem}
\label{Section:JRP}
Let us consider a Riemann Problem consisting in a Cauchy problem with constant initial data on each canal. Given a junction with $n$ incoming canals and $m$ outgoing canals, where the conserved quantities on the incoming ones are indicated by
\begin{equation}
    (x,t) \in \mathcal{I}_i \times \R^+ \mapsto u_i(x,t) \in \R \times \R, \qquad i \in \left\{1, \dots, n \right\}
\end{equation}
whereas those on the outgoing ones
\begin{equation}
    (x,t) \in \mathcal{I}_j \times \R^+ \mapsto u_j(x,t) \in \R \times \R, \qquad j \in \left\{1, \dots, m \right\}.
\end{equation}
Given constant initial conditions $(u^0_i, u^0_j)$, the Riemann solution consists of intermediate states $(u^*_i, u^*_j)$ satisfying some additional assumptions at the junction. In \cite{brianipiccoli2016} the analysis is restricted to \textit{subcritical states}, in order to provide a good definition of the solution on the network directly from the original system \eqref{shallow_system}. The main goal here is to extend the analysis to the \textit{torrential state} (or \textit{supercritical state}), by using the kinetic formulation of the system \eqref{Relaxation_System}.\\
The solution is determined such that for each $i$ and $j$ the intermediate states $(u^*_i, u^*_j)$ are connected to the initial condition through physically correct waves and in particular we have
\begin{equation}
    \Phi_\ell(h_i^*, q_i^*) = 0, \qquad \Phi_r(h_j^*, q_j^*) = 0,
\end{equation}
where $\Phi_\ell$ and $\Phi_r$ are defined in \eqref{Phi_functions}. This condition is essential to have admissible solutions to Riemann problems, since we need that $(u^0_i, u^*_i)$ is solved by $l$-waves (namely, with \textit{negative speed}) and $(u^0_j, u^*_j)$ is solve by $r$-waves (namely, with \textit{positive speed}).\\

\subsection{Bottleneck}
The simplest application of models for flow on canals is given by the bottleneck, in which we assume to have two canals intersecting at one single point. The incoming canal is parametrized by $(-\infty, 0)$ and the outgoing by $(0, +\infty)$, while the separation point (i.e. the junction) is taken in $x=0$.\\
We consider two different flux function along the canal, assuming the conservation of mass and equal heights at the junction. The homogeneous system could be explicitly rewritten as
\begin{equation}
    \begin{aligned}
    \partial_t u_1 + \partial_x F_1(u_1) = 0, \qquad &\text{for } x <0,\\ 
    \partial_t u_2 + \partial_x F_2(u_2) = 0, \qquad &\text{for } x >0,
    \end{aligned}
\end{equation}
where $F_1(u)$ and $F_2(u)$ are defined as
\begin{equation}
    F_i(u) = \begin{pmatrix}
    h v\\[5pt] \alpha_i \, h\,v^2 + \displaystyle\frac{1}{2} g\,h^2
    \end{pmatrix}.
\end{equation}
The system is complemented with junction conditions, in order to have 
\begin{equation}
\begin{aligned}
    h_1(0^-,t) = h_2(0^+,t),\\
    q_1(0^-,t) = q_2(0^+,t).
\end{aligned}
\end{equation}
At each time $t$ the solution at the interface $x=0$ is then computed by solving the nonlinear system
\begin{equation}
\label{eq::bottleneck_continuity_pressure}
    \begin{cases}
        q_1^* = q_2^*,\\
        h_1^*=h_2^*,\\
        \Phi_\ell(h_1^*, q_1^*) = 0,\\
        \Phi_r(h_2^*, q_2^*) = 0,
    \end{cases}
\end{equation}
with $h_i^*>0$. For the case of simple junctions, the interesting setting is two canals having different sections $\alpha_i$, for which the first condition reads $\alpha_1 \, h^*_1\,v^*_1 = \alpha_2 \, h_2^*\,v_2^*$.\\
This is not the only possible set of conditions to find a physically correct solution. An analogous system could be set up by assuming the energy continuity at the junction, instead of the height equality. Thus, the nonlinear system becomes
\begin{equation}
\label{eq::bottleneck_continuity_energy}
    \begin{cases}
        q_1^* = q_2^*,\\
        g\,h_1^* + \displaystyle \frac{1}{2} (v_1^*)^2=g\,h_2^*  + \displaystyle \frac{1}{2} (v_2^*)^2 ,\\
        \Phi_\ell(h_1^*, q_1^*) = 0,\\
        \Phi_r(h_2^*, q_2^*) = 0.
    \end{cases}
\end{equation}
In the following, for the sake of simplicity, we will assume the continuity of (\textit{water} or \textit{arterial}) pressure \eqref{eq::bottleneck_continuity_pressure}, instead of the energy continuity \eqref{eq::bottleneck_continuity_energy}.

\subsection{One Incoming and Two Outgoing Canals}
In this subsection we consider the particular case of a junction with one incoming and two outgoing canals. We consider an identical flux $F(u)$ on each canal and we assume again that the intersection is located in $x=0$, where the incoming canal is parametrized by $(-\infty, 0)$ and the two outgoing canals by $(0, +\infty)$. The system reads as follows
\begin{equation}
    \begin{aligned}
    \partial_t u_1 + \partial_x F(u_1) = 0, \qquad &\text{for } x <0,\\ 
    \partial_t u_2 + \partial_x F(u_2) = 0, \qquad &\text{for } x >0,\\
    \partial_t u_3 + \partial_x F(u_3) = 0, \qquad &\text{for } x >0
    \end{aligned}
\end{equation}
In order to determine the value of $u^*$ at the junction we impose again the conservation of mass and equal heights in $x=0$, namely
\begin{equation}
\begin{aligned}
    h_1(0^-,t) = h_2(0^+,t)=h_3(0^+,t),\\
    q_1(0^-,t) = q_2(0^+,t)+q_3(0^+,t).
\end{aligned}
\end{equation}
In this setting, at each time $t$, the solution at the interface $x=0$ is then computed by solving the nonlinear system
\begin{equation}
    \begin{cases}
        q_1^* = q_2^*+q_3^*,\\
        p_1^*=p_2^*=p_3^*,\\
       \Phi_\ell(h_1^*, q_1^*)=0,\\
        \Phi_r(h_2^*, q_2^*)=0,\\
        \Phi_r(h_3^*, q_3^*)=0,
    \end{cases}
\end{equation}
with $h_i^*>0$ and the $\Phi_{\{\ell, r\}}$ functions chosen as in \eqref{Phi_functions}.

\subsection{Two Incoming and One Outgoing Canals}
Let us consider a junction with two incoming canals and one outgoing canal. The system of equations for an identical flux $F(u)$ on each canal becomes now
\begin{equation}
    \begin{aligned}
    \partial_t u_1 + \partial_x F(u_1) = 0, \qquad &\text{for } x <0,\\ 
    \partial_t u_2 + \partial_x F(u_2) = 0, \qquad &\text{for } x >0,\\
    \partial_t u_3 + \partial_x F(u_3) = 0, \qquad &\text{for } x >0
    \end{aligned}
\end{equation}
The system is complemented with the assumptions on the conservation of mass and continuity of the heights at the junction, obtaining a solution to the Riemann problem consisting of three separated waves and satisfying the following nonlinear system
\begin{equation}
    \begin{cases}
        q_1^* + q_2^*=q_3^*,\\
        p_1^*=p_2^*=p_3^*,\\
        \Phi_\ell(h_1^*, q_1^*)=0,\\
        \Phi_\ell(h_2^*, q_2^*)=0,\\
        \Phi_r(h_3^*, q_3^*)=0,
    \end{cases}
\end{equation}
with $h_i^*>0$ and the $\Phi_{\{\ell, r\}}$ functions chosen as in \eqref{Phi_functions}.

\section{Numerical Simulations}
\label{Section:Numerical_Simulations}
In this section we present some numerical simulations performed with the discrete kinetic scheme previously introduced. We first examine the accuracy of the proposed scheme for a smooth solution, considering a virtual junction in the center of the domain.
All the numerical tests are performed running the scheme detailed in Section \ref{Section:Discrete_Kinetic} with $\Delta x$ precised for each test and a CFL condition of the form $\Delta t \leq 0.8 \Delta x / \lambda_{\max}$, where $\lambda_{\max}$ is the maximum absolute value of the eigenvalues of the problem. The domain, the final time $T$ and the initial conditions will be specified for each test.

\subsection{The Shallow Water Equations}

\subsubsection{Accuracy Test}
The first test is a sanity check for the numerical scheme, based on a Riemann problem for $x \in (-4, 4)$ with a virtual junction in $x=0$. More precisely, we assume the following initial condition
\begin{equation}
\label{Initial_Condition_Test1}
    \begin{aligned}
        \text{Canal $\ell$: } &h_\ell(x,0)=1, \quad &q_\ell(x,0)= 0.1,\\
        \text{Canal $r$: } &h_r(x,0)=0.5, \quad &q_r(x,0)= 0.
    \end{aligned}
\end{equation}
\begin{figure}[ht!]
    \begin{minipage}[h]{0.49\linewidth}
    \begin{center}
    \includegraphics[width=0.95\linewidth]{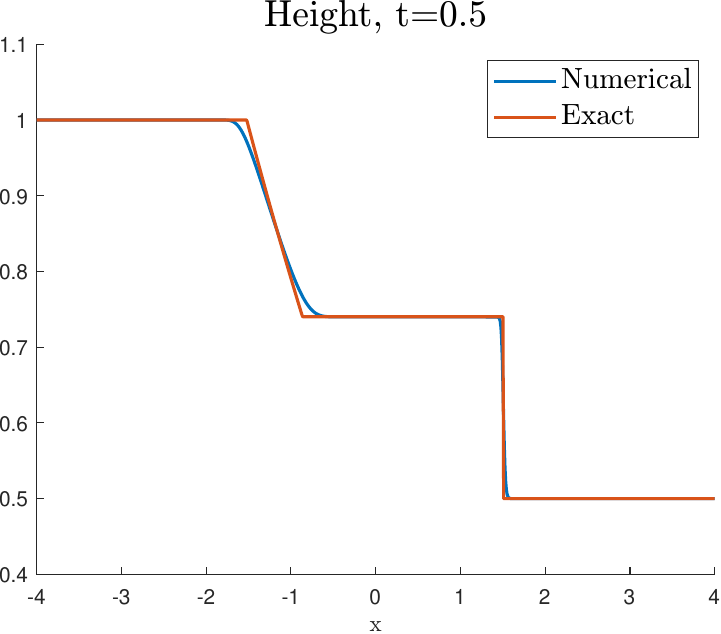} 
    \end{center}
    \end{minipage}
    \hfill
    \begin{minipage}[h]{0.49\linewidth}
    \begin{center}
    \includegraphics[width=0.95\linewidth]{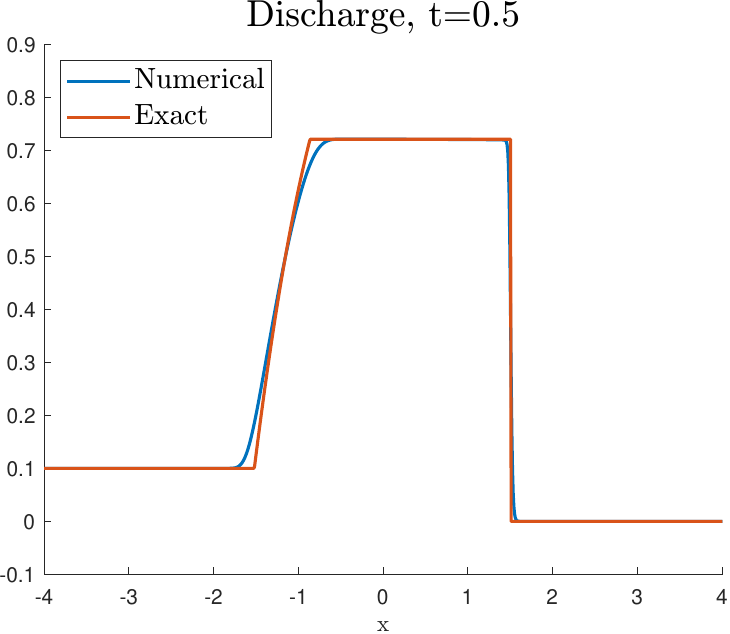} 
    \end{center} 
    \end{minipage}
    \caption{(\textit{Sanity Check}) Riemann problem \eqref{Initial_Condition_Test1} in $x \in (-4,4)$ with a virtual junction in $x=0$.}
    \label{Test_Bottleneck_Sanity}
\end{figure}
In Fig. \ref{Test_Bottleneck_Sanity} we show the numerical simulation at time $t=0.5$, compared with the exact solution analytically computed. We also test the order of convergence of the scheme, as shown in Table \ref{tab:order_accuracy}. The accuracy error is measured by the $L^1$-norm of the error between the numerical solution and the exact solution at the final time, namely
\begin{equation}
    \text{err}_{w, \ell} = \Delta x \displaystyle \sum_{j=1}^{N_\ell} \left|w^N_{j,\ell} - w^{\text{exact}}_{j,\ell} \right|,
\end{equation}
for the arc $\ell$, which could be analogously extended for the arc $r$. 

\begin{table}[ht!]
    \centering
    \begin{tabular}{|c|c c|c c|c c|c c|}
        \hline
        & \multicolumn{2}{c|}{$\mathbf{h_\ell}$} & \multicolumn{2}{c|}{$\mathbf{h_r}$} & 
        \multicolumn{2}{c|}{$\mathbf{q_\ell}$} & \multicolumn{2}{c|}{$\mathbf{q_r}$} \\
        \hline
         $\mathbf{\Delta x}$ & \textbf{err}$_{h,\ell}$ & \textbf{order} & \textbf{err}$_{h,\ell}$ & \textbf{order} & \textbf{err}$_{h,\ell}$ & \textbf{order} & \textbf{err}$_{h,\ell}$ & \textbf{order}\\
        \hline
        0.0625 & 0.0756 & // & 0.0534 & // & 0.1901 & // & 0.1696 & // \\
        0.0312 & 0.0488 & 0.6309 & 0.0286 & 0.9042 & 0.1217 & 0.6427 & 0.0904 & 0.9074 \\
        0.0156 & 0.0303 & 0.6860 & 0.0143 & 1.0007 & 0.0750 & 0.6983 & 0.0450 & 1.0083 \\
        0.0078 & 0.0181 & 0.7465 & 0.0068 & 1.0744 & 0.0444 & 0.7566 & 0.0210 & 1.0988 \\
        0.0039 & 0.0102 & 0.8327 & 0.0029 & 1.2201 & 0.0240 & 0.8447 & 0.0089 & 1.2444 \\
        0.002 & 0.0051 & 0.9827 & 0.00098 & 1.5594 & 0.0123 & 1.0034 & 0.0029 & 1.6145 \\
        \hline
    \end{tabular}
    \caption{$L^1$ error and order of convergence of the discrete kinetic scheme on the left arc (denoted by $\ell$) and the right arc (denoted by $r$) at final time $T=1$.}
    \label{tab:order_accuracy}
\end{table}
\subsubsection{Bottleneck: fluvial regime}
In this subsection we assume to have only two canals with different cross sections, connected at one point in the center of the spatial domain. In all tests we assume constant in time width along each canal.\\
We consider the same initial condition in \eqref{Initial_Condition_Test1}, with two different cross sections, namely setting $\alpha_\ell=1$ and $\alpha_r=0.5$, where $\alpha_i$ is the width of the $i$-th canal. In Fig. \ref{Test2_Bottleneck} we show again that the discrete kinetic approximation is able to capture the behavior of the solution at the junction. We compare the dynamics obtained by the scheme with two different junction conditions: the relaxation approach, based on the exact solution of the linear Riemann problem at the junction (see Section \ref{Treatment_Junction}) and the classic approach, based on an approximate Riemann solver for the solution of system (see Appendix \ref{Appendix:Solution_Riemann}).\\
\begin{figure}[ht!]
    \begin{minipage}[h]{0.49\linewidth}
    \begin{center}
    \includegraphics[width=0.95\linewidth]{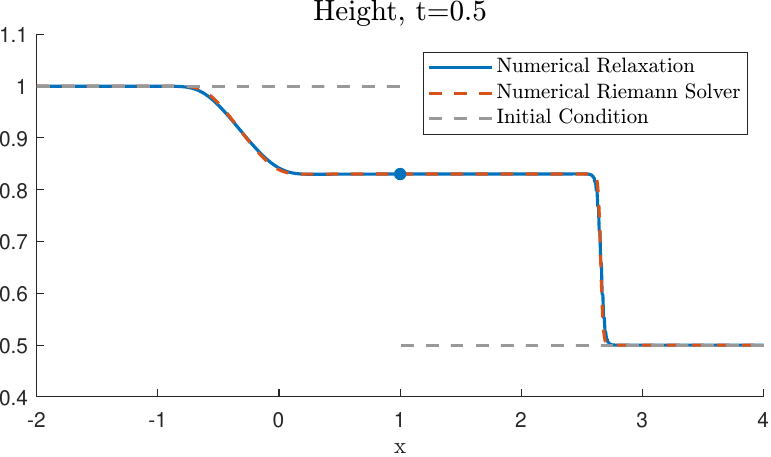} 
    \end{center}
    \end{minipage}
    \hfill
    \begin{minipage}[h]{0.49\linewidth}
    \begin{center}
    \includegraphics[width=0.95\linewidth]{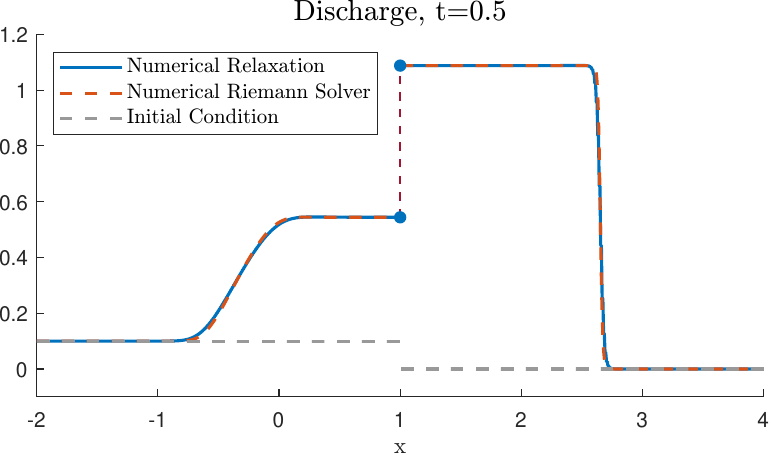} 
    \end{center} 
    \end{minipage}
    \caption{(\textit{Bottleneck}) Riemann problem \eqref{Initial_Condition_Test1} with two different cross sections in $x \in (-2,4)$ with a junction in $x=1$.}
    \label{Test2_Bottleneck}
\end{figure}

Finally, we test the bottleneck for a non-constant initial state in $x \in (-2, 4)$ where the junction is located at $x=1$, with the following initial condition
\begin{equation}
\label{Initial_Condition_Test3}
    \begin{aligned}
        \text{Canal $\ell$: } &h_\ell(x,0)= 1+\exp(-20 x^2), \quad &q_\ell(x,0)= 0.5 h_\ell(x,0),\\
        \text{Canal $r$: } &h_r(x,0)=1, \quad &q_r(x,0)= 0.1.
    \end{aligned}
\end{equation}
We compare again the dynamics obtained by the scheme with the two different junction conditions at final time $t=0.12$, computed with $\Delta x = 0.006$.
\begin{figure}[ht!]
    \begin{minipage}[h]{0.49\linewidth}
    \begin{center}
    \includegraphics[width=0.95\linewidth]{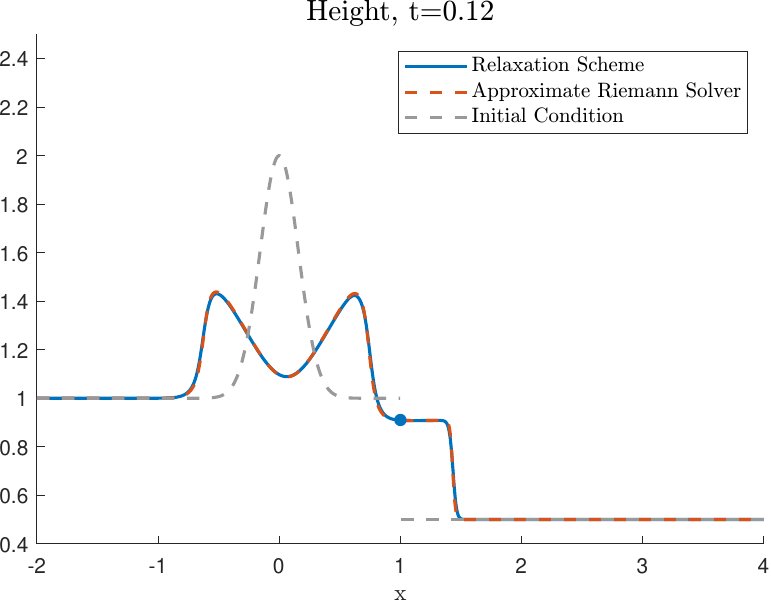} 
    \end{center}
    \end{minipage}
    \hfill
    \begin{minipage}[h]{0.49\linewidth}
    \begin{center}
    \includegraphics[width=0.95\linewidth]{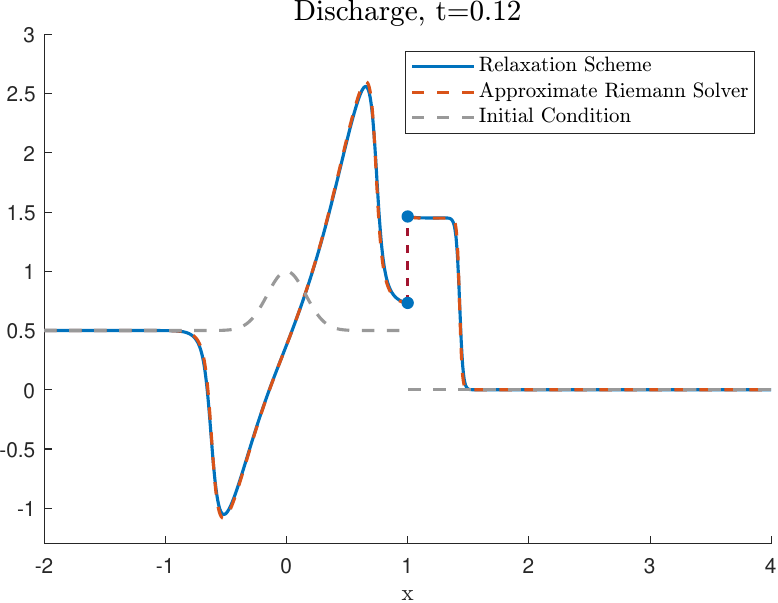} 
    \end{center} 
    \end{minipage}
    \caption{(\textit{Bottleneck}) Riemann problem \eqref{Initial_Condition_Test3} in $x \in (-2,4)$ with a junction in $x=1$.}
    \label{Test3_Bottleneck}
\end{figure}

\subsubsection{Bottleneck: fluvial to torrential regime}
In this subsection we assume again only two canals, but we consider an initial condition which leads to a transition from fluvial to torrential regime. We set $\Delta x = 0.004$ for all the numerical tests. Following \cite{brianipiccoli2018}, let us set the following initial condition
\begin{equation}
\label{SW_Bottleneck_FT}
    \begin{aligned}
        \text{Canal $\ell$: } &h_1\ell(x,0)=0.25, \quad &q_\ell(x,0)= 0.025,\\
        \text{Canal $r$: } &h_r(x,0)=2.5, \quad &q_r(x,0)= 0.25.
    \end{aligned}
\end{equation}

Despite both initial states are subcritical ($Fr_\ell \approx 0.06$ and $Fr_r \approx 0.02$), the intermediate state (or region) can reach a supercritical value ($Fr >1$). Indeed, since the ratio $h_r/h_\ell >> 1$, the rarefaction fan on the left can accelerate the fluid to supercritical velocity \cite{brianipiccoli2018}.  
In Fig. \ref{Test4_Bottleneck} we compare our approach with the strategy proposed in \cite{brianipiccoli2018}. As observed, there is perfect agreement between the two solutions.

\begin{figure}[ht!]
    \begin{minipage}[h]{0.49\linewidth}
    \begin{center}
    \includegraphics[width=0.95\linewidth]{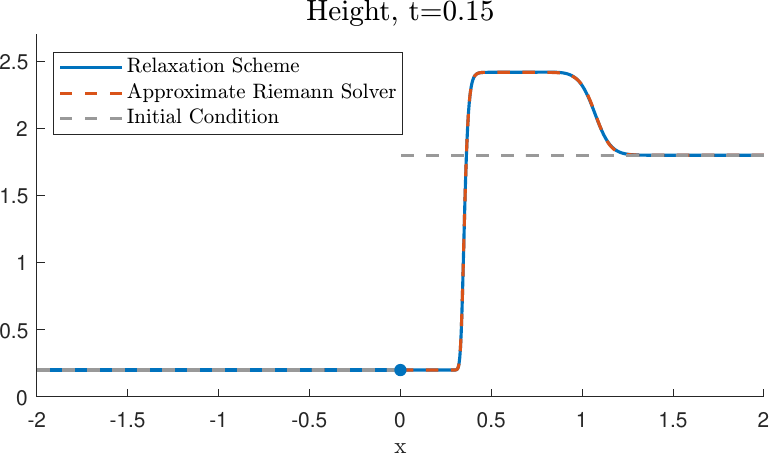} 
    \end{center}
    \end{minipage}
    \hfill
    \begin{minipage}[h]{0.49\linewidth}
    \begin{center}
    \includegraphics[width=\linewidth]{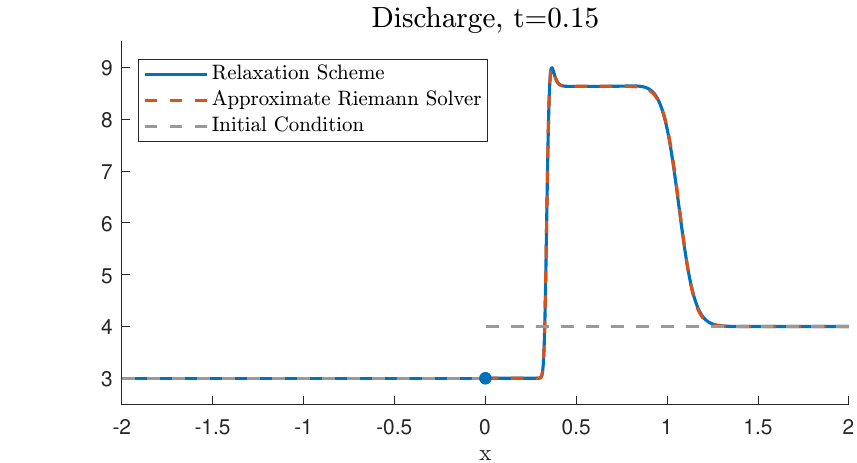} 
    \end{center} 
    \end{minipage}
    \caption{(\textit{Bottleneck}) Riemann problem \eqref{SW_Bottleneck_FT} in $x \in (-2,2)$ with a junction in $x=0$ for a supercritical case.}
    \label{Test4_Bottleneck}
\end{figure}

Then, we consider the following initial condition
\begin{equation}
\label{SW_Bottleneck_FT2}
    \begin{aligned}
        \text{Canal $\ell$: } &h_\ell(x,0)=0.2, \quad &q_\ell(x,0)= 3,\\
        \text{Canal $r$: } &h_r(x,0)=1.8, \quad &q_r(x,0)= 4,
    \end{aligned}
\end{equation}
in which the left state is supercritical ($Fr_\ell >1$) and the right state is subcritical ($Fr_r <1 $).\\
Figure \ref{Test5_Bottleneck} shows again the comparison between two different schemes at final time $T=0.15$. The results are in perfect agreement also for this initial condition, showing the capability of relaxation schemes to deal also supercritical initial conditions.

\begin{figure}[ht!]
    \begin{minipage}[h]{0.49\linewidth}
    \begin{center}
    \includegraphics[width=0.95\linewidth]{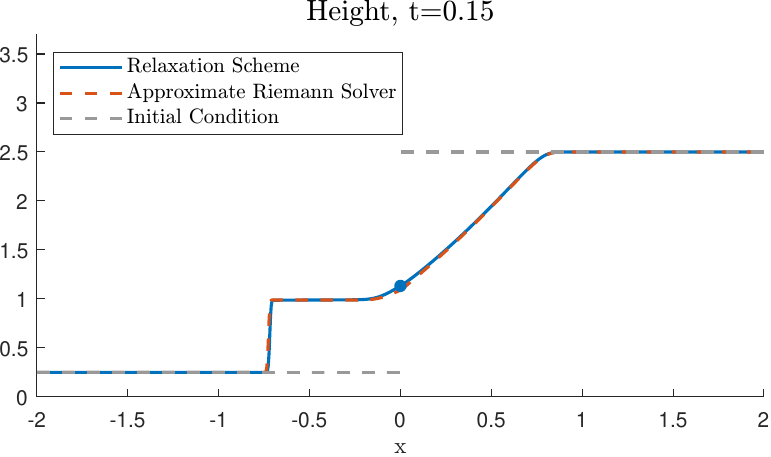} 
    \end{center}
    \end{minipage}
    \hfill
    \begin{minipage}[h]{0.49\linewidth}
    \begin{center}
    \includegraphics[width=0.95\linewidth]{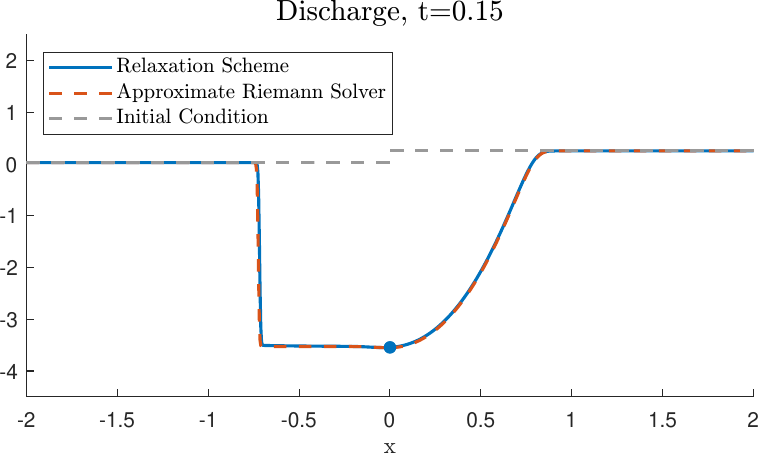} 
    \end{center} 
    \end{minipage}
    \caption{(\textit{Bottleneck}) Riemann problem \eqref{SW_Bottleneck_FT2} in $x \in (-2,2)$ with a junction in $x=0$ for a supercritical case.}
    \label{Test5_Bottleneck}
\end{figure}

\subsubsection{Bottleneck: well-balanced property}
Now, we consider the shallow water equations with a source term in presence of a bottleneck, in order to study the well-balanced property of the numerical scheme. Let us consider as initial condition, a perturbation of the steady state of a ``lake at rest'' with varying bottom topography, already proposed in \cite{leveque1998, audusse2004, noelle2006}
\begin{equation}
    z(x) = \begin{cases}
        0.25(1+\cos(10\pi(x-0.5))), \qquad &\text{if } \, 1.2 \leq x \leq 1.4,\\
        0, \qquad &\text{otherwise},
    \end{cases}
\end{equation}
where $x \in [0,2]$ and setting $\Delta x =0.04$. Since the goal of this test is to investigate the well-balanced property of the scheme, for the sake of simplicity we assume again a virtual junction in $x=1$, without variation in the cross sections of the canals. The initial condition is then given by
\begin{equation}
\label{SW_Bottleneck_WB}
    \begin{aligned}
        \text{Canal $\ell$: } &h_\ell(x,0)=0, \quad &q_\ell(x,0)= 0,\\
        \text{Canal $r$: } &h_r(x,0)=\begin{cases}
            1 + \Delta H \qquad &\text{if } \, 1.2 \leq x \leq 1.4,\\
            1, \qquad &\text{otherwise},
        \end{cases}, \quad &q_r(x,0)= 0,
    \end{aligned}
\end{equation}
with $\Delta H = 0.001$, to obtain a very small perturbation, which is more challenging to capture. Figures \ref{TestWB_Bottleneck} show a comparison of the height and the momentum computed between the numerical solution obtained by considering the relaxation scheme with a virtual junction in $x=1$ and the numerical solution obtained on the full interval $[0, 2]$, without junction conditions. 

\begin{figure}[ht!]
    \begin{minipage}[h]{0.49\linewidth}
    \begin{center}
    \includegraphics[width=0.95\linewidth]{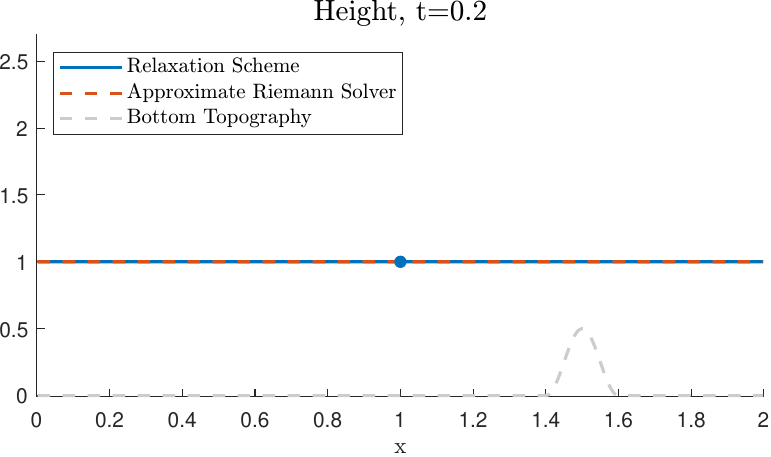} 
    \end{center}
    \end{minipage}
    \hfill
    \begin{minipage}[h]{0.49\linewidth}
    \begin{center}
    \includegraphics[width=0.95\linewidth]{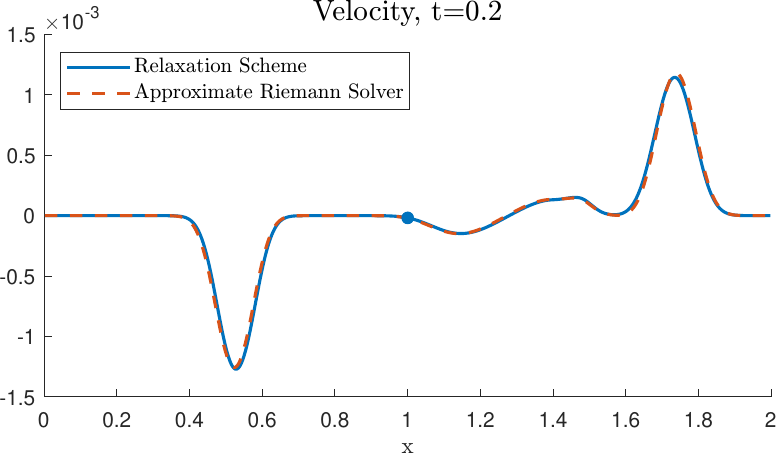} 
    \end{center} 
    \end{minipage}
    \caption{(\textit{Bottleneck}) Perturbation of \textit{lake at rest} steady state \eqref{SW_Bottleneck_WB} in $x \in (0,2)$ with a junction in $x=1$.}
    \label{TestWB_Bottleneck}
\end{figure}

We also consider as a test case a steady state with constant free surface and discontinuous bottom topography at the junction, namely
\begin{equation}
\label{eq::initial_discontinuous_topography}
    z(x) = \begin{cases}
        4, \qquad &\text{if } x \leq 1,\\
        0, \qquad &\text{otherwise},
    \end{cases}
\end{equation}
In Fig. \ref{TestWB_Bottleneck_Discontinuous} we observe that the constant steady state is preserved over time.

\begin{figure}[ht!]
    \begin{center}    \includegraphics[width=0.55\linewidth]{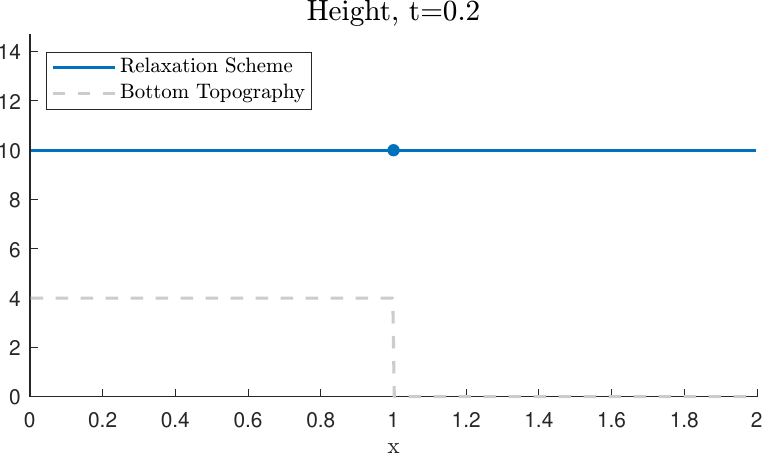} 
    \end{center} 
    \caption{(\textit{Bottleneck}) Constant steady state in $x \in (0,2)$ with junction in $x=1$ and discontinuous bottom topography \eqref{eq::initial_discontinuous_topography} 
    (gray).}
    \label{TestWB_Bottleneck_Discontinuous}
\end{figure}

\subsubsection{One Incoming - Two Outgoing Canals}
This subsection is devoted to the case of more realistic networks, consisting of a junction with one incoming and two outgoing canals. Let us fix the cross section $\alpha=1$ for all canals. We denote by subscript $1$ the incoming canal and by subscripts $2$ and $3$ the outgoing ones. The first initial condition is the following
\begin{equation}
\label{SW_1-2_Test0}
    \begin{aligned}
        \text{Canal $1$: } &h_1(x,0)=0.5, \quad &q_1(x,0)= 0.1,\\
        \text{Canal $2$: } &h_2(x,0)=0.5, \quad &q_2(x,0)= 0,\\
        \text{Canal $3$: } &h_3(x,0)=1, \quad &q_3(x,0)= 0.
    \end{aligned}
\end{equation}
In Figure \ref{Test0_OneInTwoOut} we observe the comparison of the solution obtained by applying the relaxation scheme and the solution obtained by using the approximate Riemann solver at the junction, which are perfectly overlapped.\\

We consider now the following initial data in $[-2,4]$, with the junction located in $x=1$, where the incoming canal is parametrized for $x \in [-2, 1]$, while the outgoing ones for $x \in [1, 4]$, as in \cite{brianipiccoli2016}
\begin{equation}
\label{SW_1-2_Test1}
    \begin{aligned}
        \text{Canal $1$: } &h_1(x,0)=1+\exp(-20\,x^2), \quad &q_1(x,0)= 0.5\, h_1(x,0),\\
        \text{Canal $2$: } &h_2(x,0)=1, \quad &q_2(x,0)= 0,\\
        \text{Canal $3$: } &h_3(x,0)=0.5, \quad &q_3(x,0)= 0.
    \end{aligned}
\end{equation}
The numerical results are presented in Figure \ref{Test1_OneInTwoOut}, where the comparison is again performed against the numerical solution computed with a Riemann solver at the junction.\\ 

We finally consider an initial condition given by supercritical states, where the solution to the Riemann problem cannot be explicitly computed. Let us consider as initial condition 
\begin{equation}
\label{SW_1-2_Test7}
    \begin{aligned}
        \text{Canal $1$: } &h_1(x,0)=0.25, \quad &q_1(x,0)= 0.5591,\\
        \text{Canal $2$: } &h_2(x,0)=0.15, \quad &q_2(x,0)= 0.2795,\\
        \text{Canal $3$: } &h_3(x,0)=0.1, \quad &q_3(x,0)= 0.1009.
    \end{aligned}
\end{equation}
In this case, we observe in Fig. \ref{Test7_OneInTwoOut} that all channels are initially supercritical ($Fr_i >1$). The incoming channel remains in the initial state and the two outgoing channels develop a rarefaction wave.

\begin{figure}[ht!]
    \begin{minipage}[h]{0.49\linewidth}
    \begin{center}
    \includegraphics[width=0.95\linewidth]{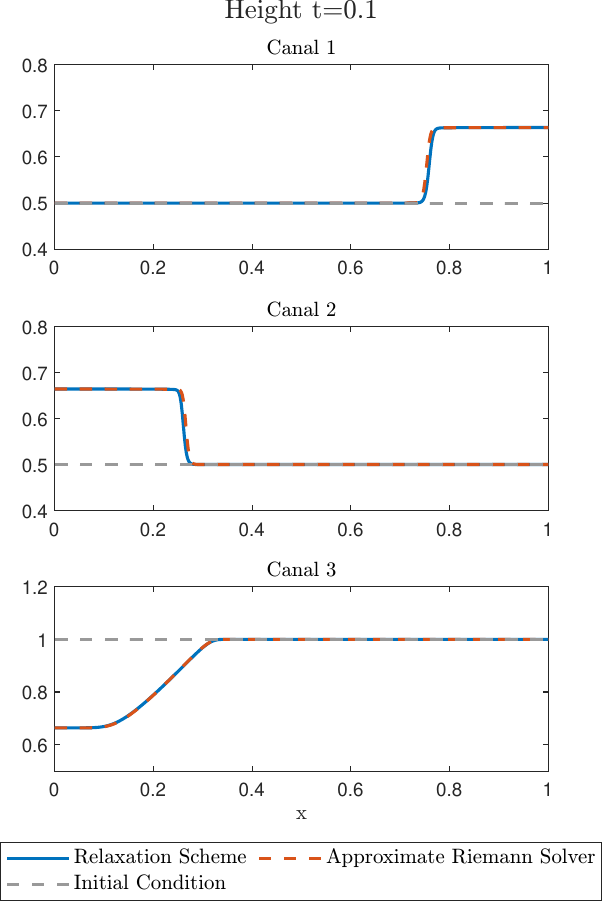} 
    \end{center}
    \end{minipage}
    \hfill
    \begin{minipage}[h]{0.49\linewidth}
    \begin{center}
    \includegraphics[width=0.95\linewidth]{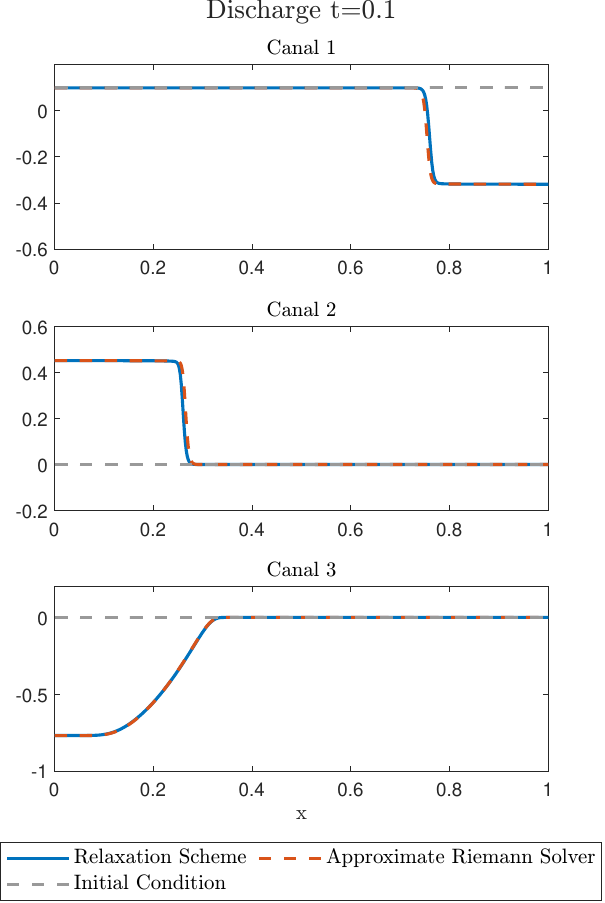} 
    \end{center} 
    \end{minipage}
    \hfill
    \vspace{0.5cm}
    \begin{minipage}[h]{0.49\linewidth}
    \begin{center}
    \includegraphics[width=0.95\linewidth]{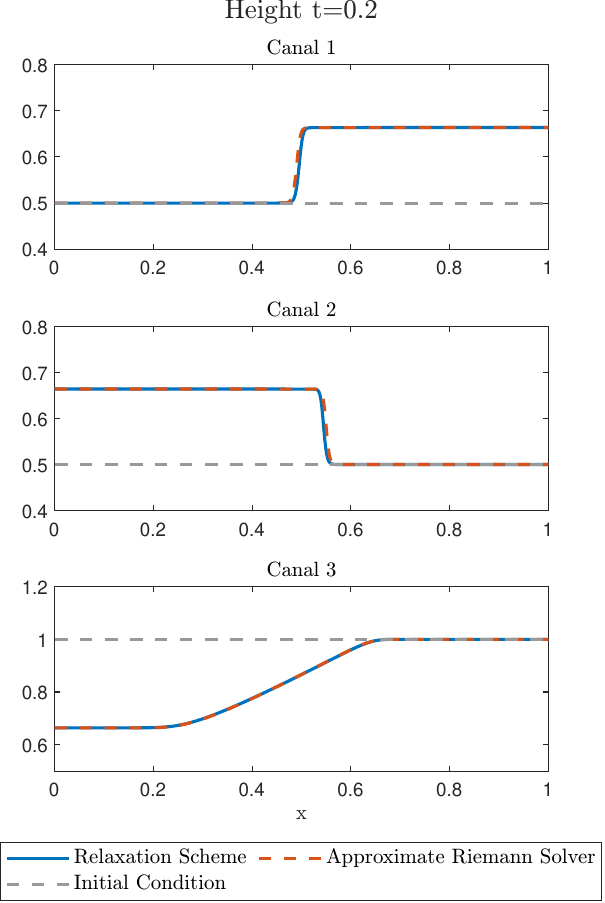} 
    \end{center}
    \end{minipage}
    \hfill
    \begin{minipage}[h]{0.49\linewidth}
    \begin{center}
    \includegraphics[width=0.95\linewidth]{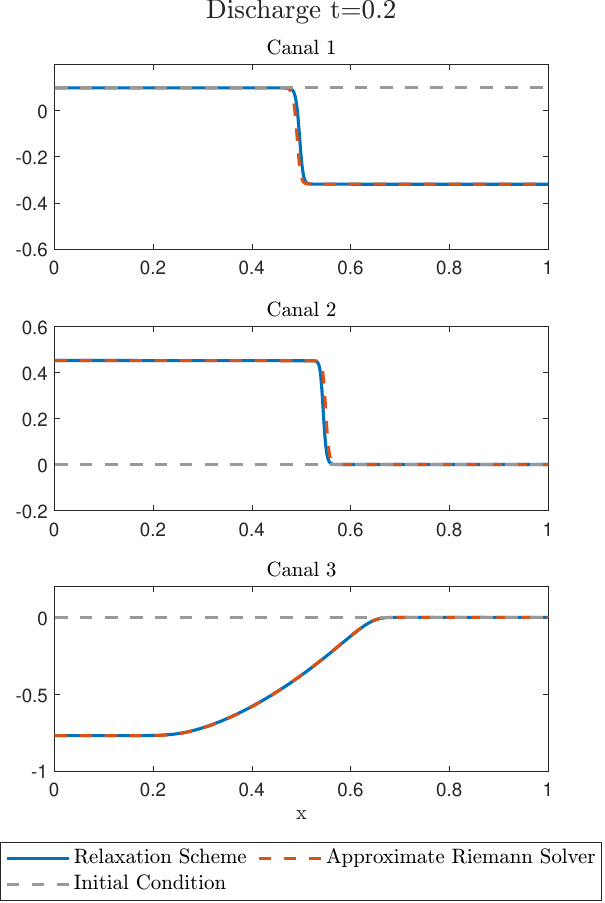} 
    \end{center} 
    \end{minipage}
    \caption{(\textit{One Incoming - Two Outgoing}) Riemann problem \eqref{SW_1-2_Test0} in $x \in (-2,4)$ with a junction in $x=1$.}
    \label{Test0_OneInTwoOut}
\end{figure}

\begin{figure}[ht!]
    \begin{minipage}[h]{0.49\linewidth}
    \begin{center}
    \includegraphics[width=0.95\linewidth]{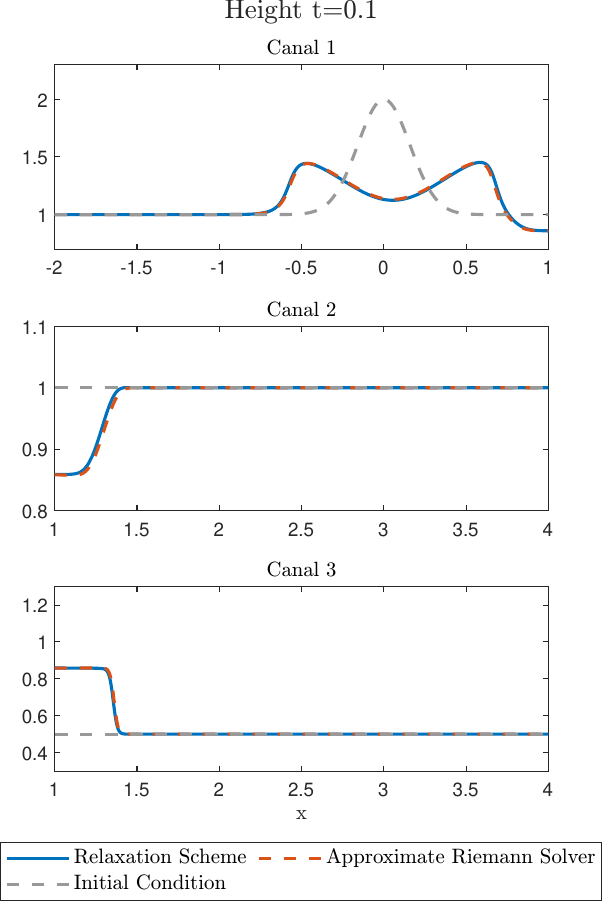} 
    \end{center}
    \end{minipage}
    \hfill
    \begin{minipage}[h]{0.49\linewidth}
    \begin{center}
    \includegraphics[width=0.95\linewidth]{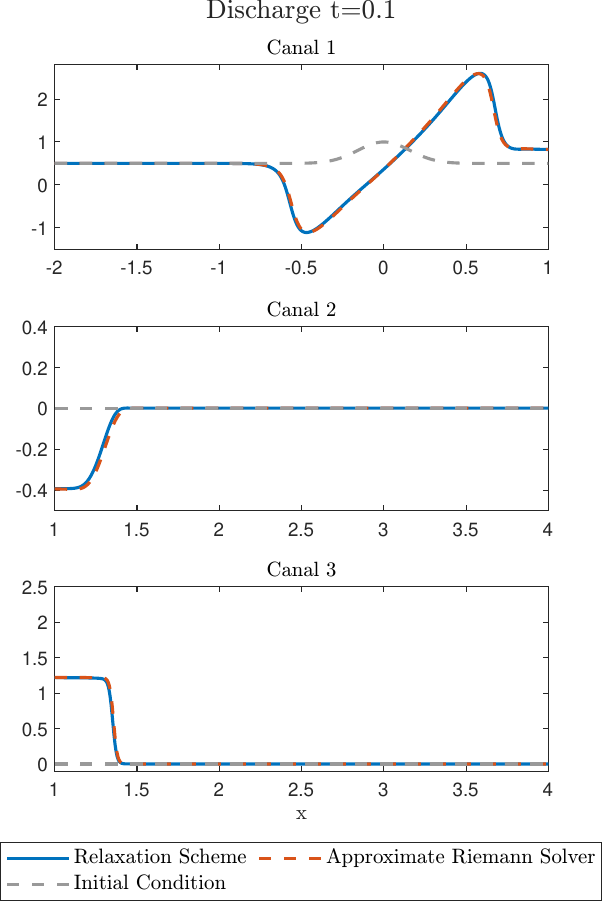} 
    \end{center} 
    \end{minipage}    
    \hfill
    \vspace{0.5cm}
    \begin{minipage}[h]{0.49\linewidth}
    \begin{center}
    \includegraphics[width=0.95\linewidth]{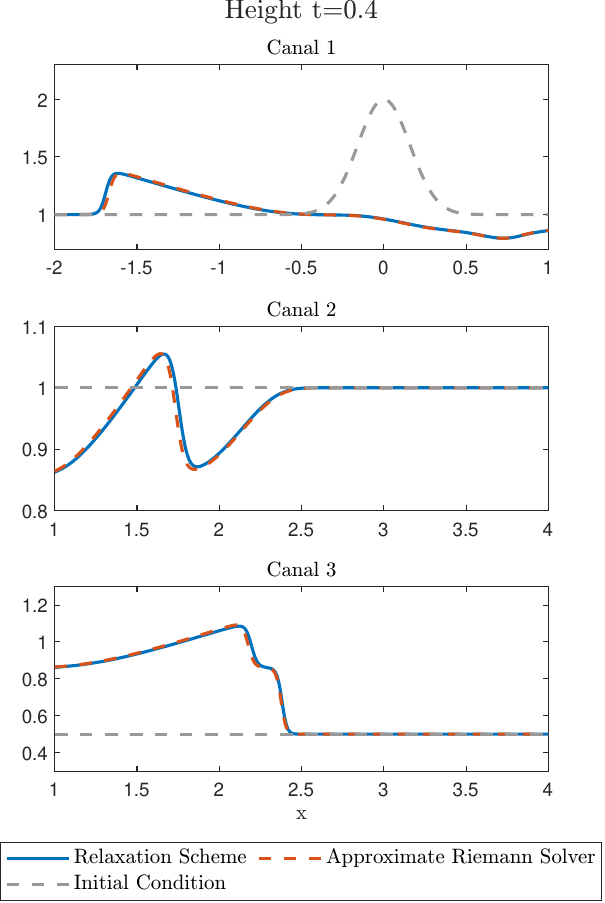} 
    \end{center}
    \end{minipage}
    \hfill
    \begin{minipage}[h]{0.49\linewidth}
    \begin{center}
    \includegraphics[width=0.95\linewidth]{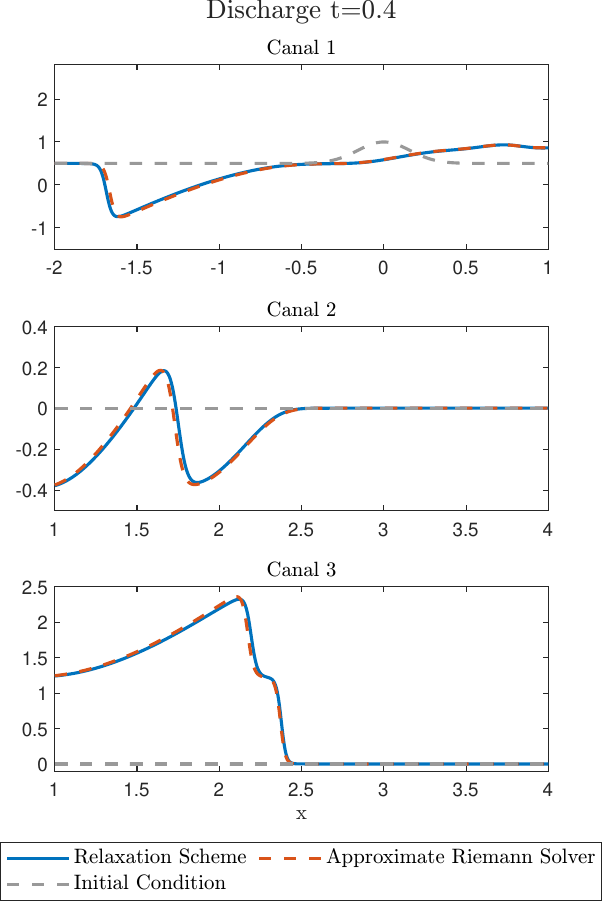} 
    \end{center} 
    \end{minipage}
    \caption{(\textit{One Incoming - Two Outgoing}) Riemann problem \eqref{SW_1-2_Test1} in $x \in (-2,4)$ with a junction in $x=1$.}
    \label{Test1_OneInTwoOut}
\end{figure}

\begin{figure}[ht!]
    \begin{minipage}[h]{0.49\linewidth}
    \begin{center}
    \includegraphics[width=0.85\linewidth]{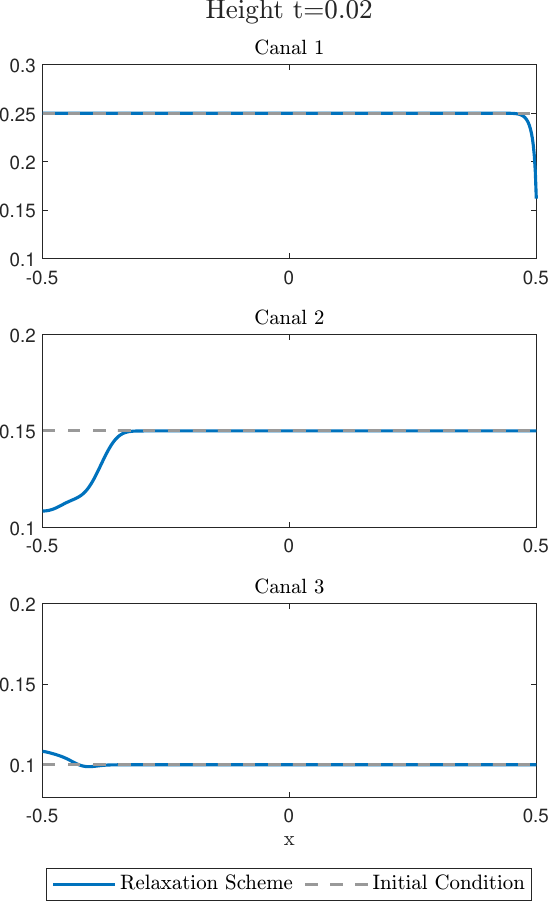} 
    \end{center}
    \end{minipage}
    \hfill
    \begin{minipage}[h]{0.49\linewidth}
    \begin{center}
    \includegraphics[width=0.85\linewidth]{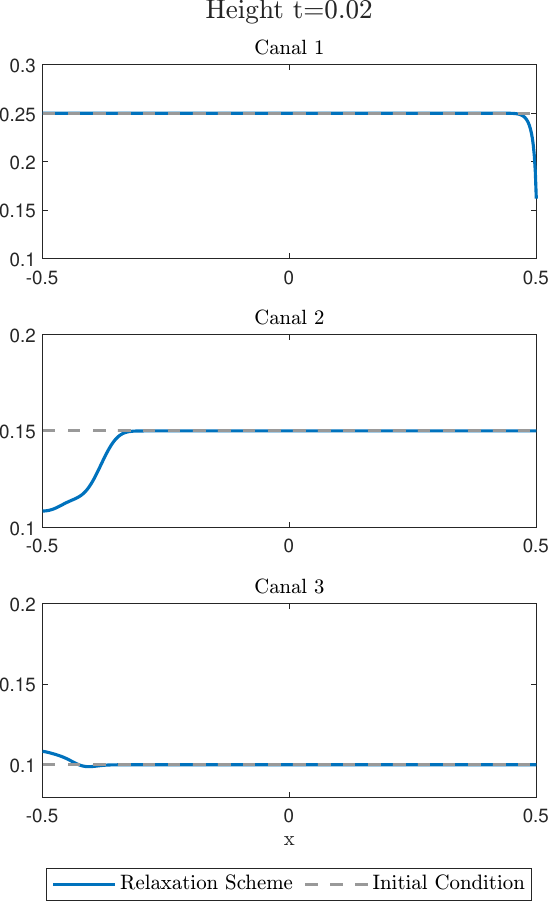} 
    \end{center} 
    \end{minipage}    
    \hfill
    \vspace{0.5cm}
    \begin{minipage}[h]{0.49\linewidth}
    \begin{center}
    \includegraphics[width=0.85\linewidth]{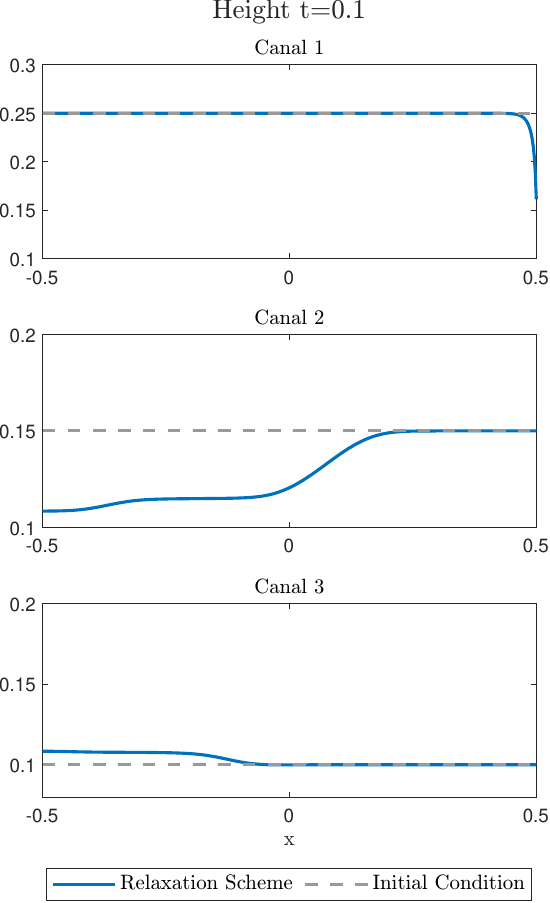} 
    \end{center}
    \end{minipage}
    \hfill
    \begin{minipage}[h]{0.49\linewidth}
    \begin{center}
    \includegraphics[width=0.85\linewidth]{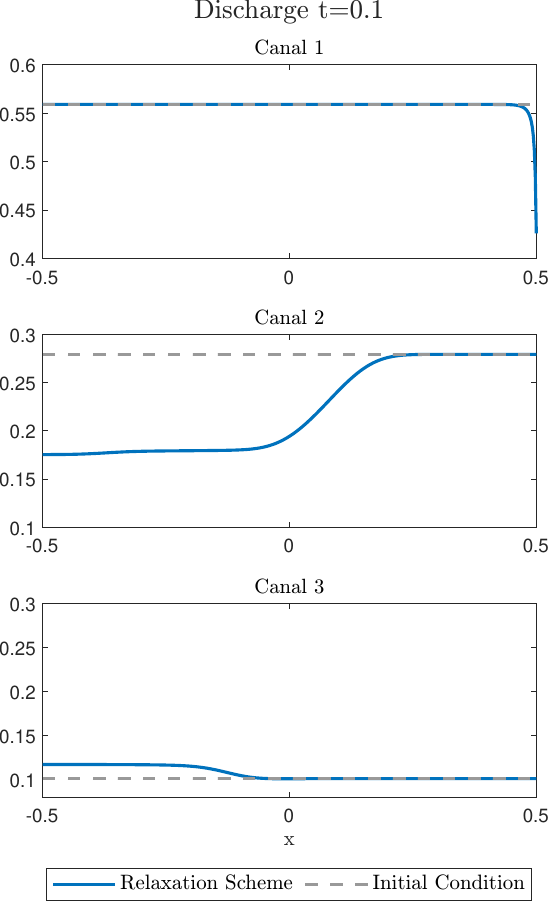} 
    \end{center} 
    \end{minipage}
    \caption{(\textit{One Incoming - Two Outgoing}) Riemann problem \eqref{SW_1-2_Test7} $x \in (-0.5,0.5)$ with a junction in $x=-0$.}
    \label{Test7_OneInTwoOut}
\end{figure}

\subsubsection{Two Incoming - One Outgoing Canals}
Let us now consider a crossing with two incoming canals and one outgoing canal, all parametrized in the same domain. The initial condition is the following \cite{brianipiccoli2016}
\begin{equation}
\label{SW_2-1_Test1}
    \begin{aligned}
        &\text{Canal 1: } &h_1(x,0)=
        \begin{cases}
            1.5 \qquad &\text{if } x \in [0, 0.2] \cup [0.4, 0.6] \cup [0.6, 0.8],\\
            1 \qquad &\text{otherwise}
        \end{cases} \\
        & &q_1(x,0)= 0.5\,h_1(x,0),\\[6pt]
        &\text{Canal 2: } &h_2(x,0)=\begin{cases}
            1.5 \qquad &\text{if } x \in [0, 0.2] \cup [0.4, 0.6] \cup [0.6, 0.8],\\
            1 \qquad &\text{otherwise}
        \end{cases} \\
        & &q_2(x,0)= 0.5\,h_2(x,0),\\[6pt]
        &\text{Canal 3: } &h_3(x,0)=1, \\ \quad & &q_3(x,0)= 0.
    \end{aligned}
\end{equation}
\begin{figure}[htbp]
    \begin{minipage}[h]{0.49\linewidth}
    \begin{center}
    \includegraphics[width=0.95\linewidth]{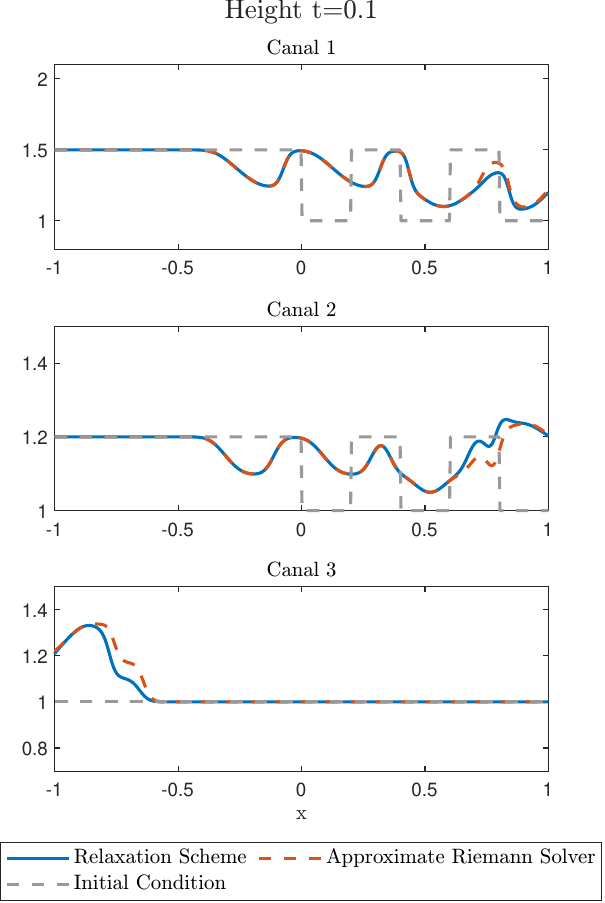} 
    \end{center}
    \end{minipage}
    \hfill
    \begin{minipage}[h]{0.49\linewidth}
    \begin{center}
    \includegraphics[width=0.95\linewidth]{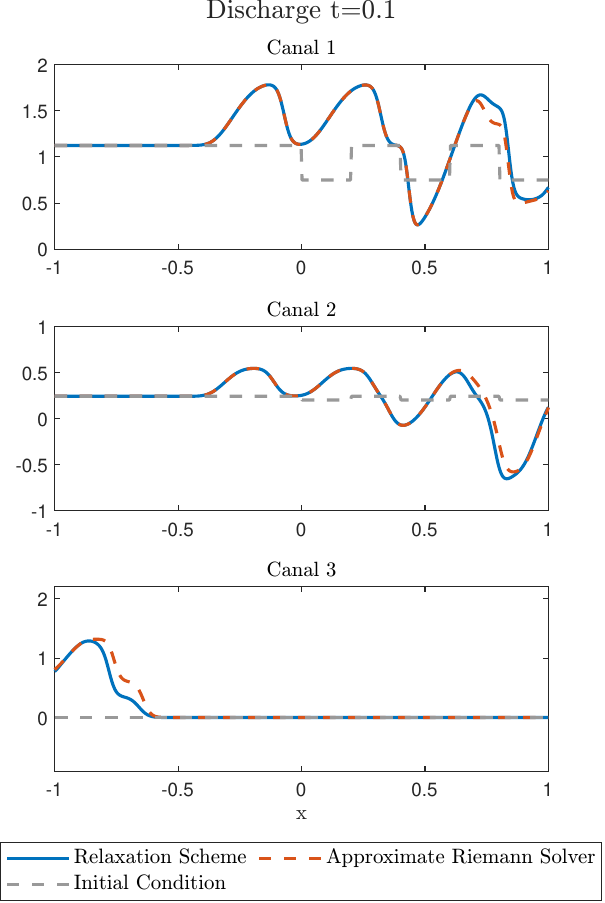} 
    \end{center} 
    \end{minipage}
    \caption{(\textit{Two Incoming - One Outgoing}) Riemann problem \eqref{SW_2-1_Test1} with $x \in (-1,1)$ and the junction in $x=0$. We show the numerical solution at time $t=0.1$ (\textit{top}).}
    \label{Test2_TwoInOneOut}
\end{figure}

\begin{figure}[htbp]\ContinuedFloat
    \begin{minipage}[h]{0.49\linewidth}
    \begin{center}
    \includegraphics[width=0.95\linewidth]{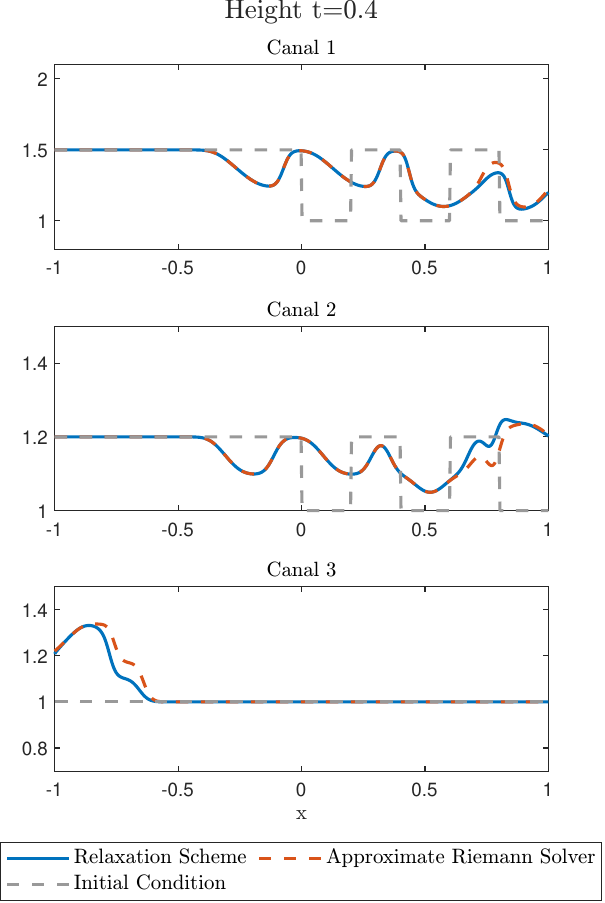} 
    \end{center}
    \end{minipage}
    \hfill
    \begin{minipage}[h]{0.49\linewidth}
    \begin{center}
    \includegraphics[width=0.95\linewidth]{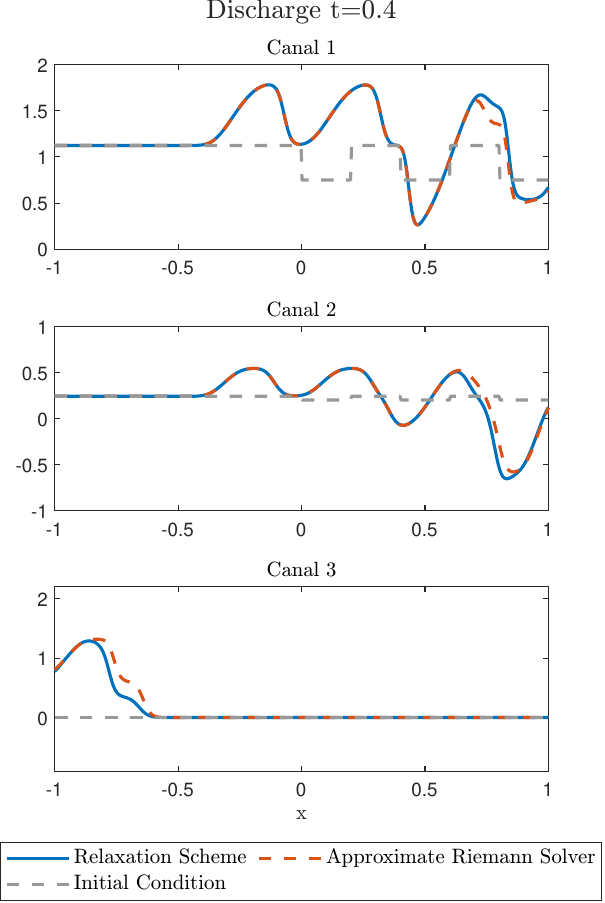} 
    \end{center} 
    \end{minipage}
    \caption{(\textit{Two Incoming - One Outgoing}) Riemann problem \eqref{SW_2-1_Test1} with $x \in (-1,1)$ and the junction in $x=0$. We show the numerical solution at time $t=0.4$.}
\end{figure}

As expected, also for this network we have good agreement between the solution computed by using relaxation schemes and the solution obtained with the approximate Riemann solver at the junction. In Fig. \ref{Test2_TwoInOneOut} we plot the solution at two different times $t=0.1$ and $t=0.4$. 

\newpage

\subsection{The Arterial Network}

\subsubsection{Bottleneck: well-balanced property}
For the blood flow equations we investigate analogously the well-balanced property, by considering a perturbation of the steady state, known as ``man at eternal rest''. Let us consider the set of parameters proposed in \cite{delestre2013},  $\nu=0$, $\beta = 1.0 \cdot 10^8$ Pa/m, $\rho=1060$ m$^3$, $L=0.14$m and the final time $T=5$ and the following initial condition with initial null velocity, namely $q_\ell(x,0) \equiv 0$ and $q_r(x,0) = 0$. We consider as initial vessel section $a_\ell(x,0) = \pi \, R_\ell(x,0)$, where
\begin{equation}
    R_\ell(x,0) = \begin{cases}
        R_0 \qquad &\text{if } \, x \in [0, x_1],\\
        R_0 + \displaystyle \frac{\Delta R}{2} \left[\sin \left(\frac{x-x_1}{x_2-x_1}\pi - \frac{\pi}{2} \right) +1 \right] \qquad &\text{if } \, x \in (x_1, x_2),\\
        R_0 + \Delta R \qquad &\text{if } \, x \in [x_2, x_3],\\
        R_0 + \displaystyle\frac{\Delta R}{2} \left[\cos \left(\frac{x-x_3}{x_4-x_3}\pi \right) +1 \right] \qquad &\text{if } \, x \in (x_3, x_c],
    \end{cases}
\end{equation}
with $R_0 = 4 \cdot 10^{-3}$, $\Delta R = 1.0 \cdot 10^{-3}$, $x_1 = 1.0 \cdot 10^{-2}$, $x_2 = 3.05 \cdot 10^{-2}$, $x_3 = 4.95 \cdot 10^{-2}$, $x_4 = 7.0 \cdot 10^{-2}$ and $x_c = L/2$ (the center of the domain). The right artery has the following initial condition
\begin{equation}
    R_\ell(x,0) = \begin{cases}
        R_0 \qquad &\text{if } \, x \in [x_4, L],\\
        R_0 + \displaystyle\frac{\Delta R}{2} \left[\cos \left(\frac{x-x_3}{x_4-x_3}\pi \right) +1 \right] \qquad &\text{if } \, x \in [x_c, x_4).
    \end{cases}
\end{equation}
In Figure \ref{TestWB_Bottleneck_Arterial} we observe that the numerical solution obtained with the well-balanced scheme described in Section \ref{Section:Well_Balanced} perfectly preserves the reference steady state.

\begin{figure}[ht!]
    \begin{minipage}[h]{0.49\linewidth}
    \begin{center}
    \includegraphics[width=0.95\linewidth]{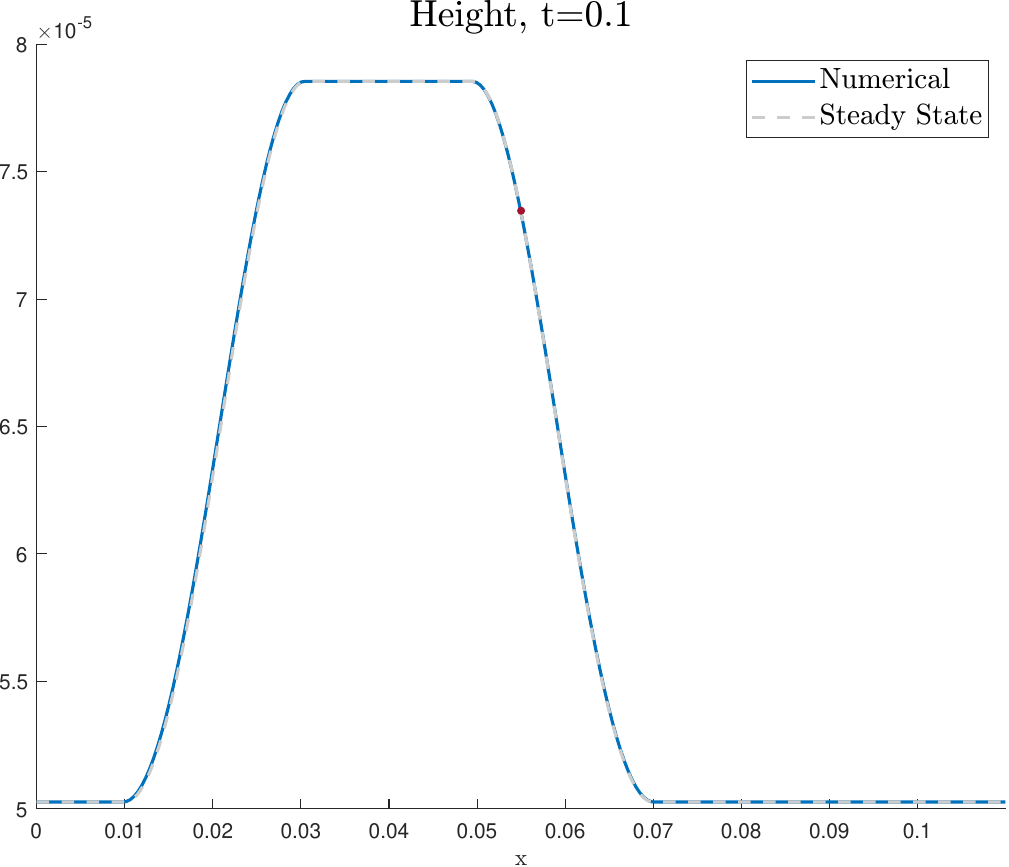} 
    \end{center}
    \end{minipage}
    \hfill
    \begin{minipage}[h]{0.49\linewidth}
    \begin{center}
    \includegraphics[width=0.95\linewidth]{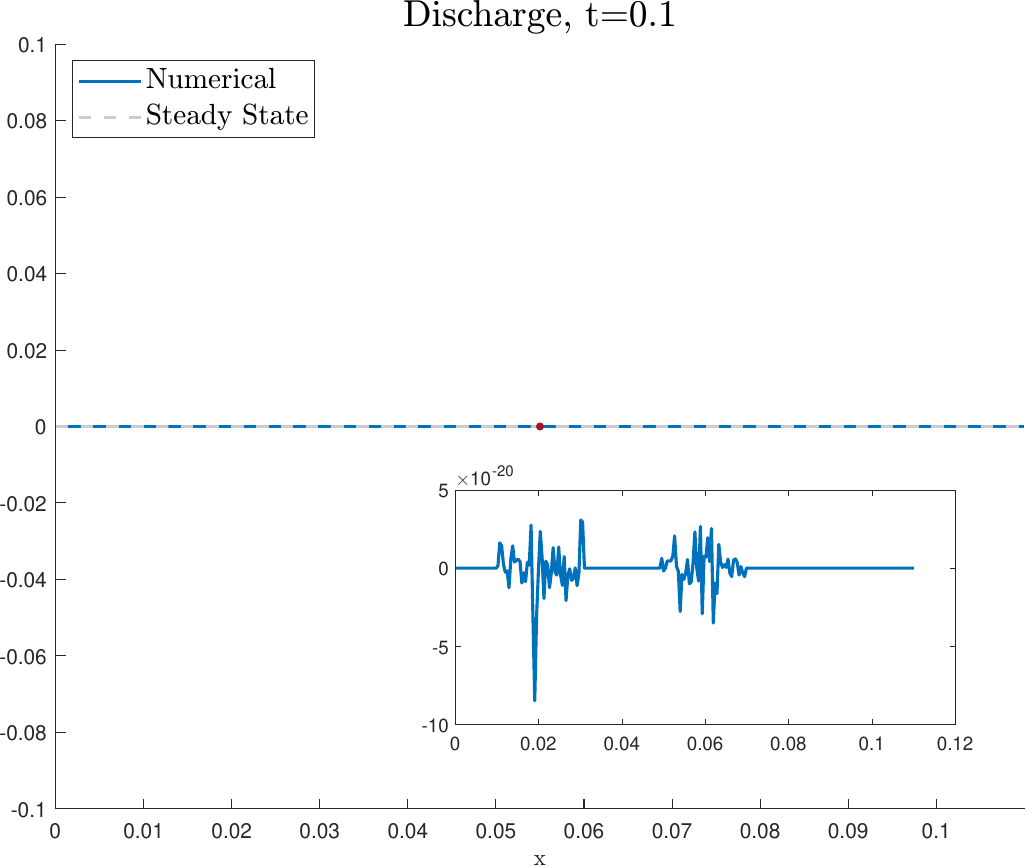} 
    \end{center} 
    \end{minipage}
    \caption{(\textit{Bottleneck}) Perturbation of \textit{man at eternal rest} steady state in $x \in (0,2)$ with a junction in $x=1$.}
    \label{TestWB_Bottleneck_Arterial}
\end{figure}

\subsubsection{One Incoming - Two Outgoing Arteries}
Let us consider a network of one incoming artery and two outgoing ones.\\ 
The first scenario we would like to reproduce is a symmetric bifurcation. This could model for instance the aortic bifurcation, in which the abdominal aorta bifurcates into the left and right common iliac arteries. Let us consider the following initial condition, inspired by \cite{boileau2015}
\begin{equation}
\label{BloodFlow_1-2_Test1}
    \begin{aligned}
        \text{Artery $1$: } &a_1(x,0)=1.8055, \quad &q_1(x,0)= 3.6110 \cdot 10^{-1},\\
        \text{Artery $2$: } &a_2(x,0)=9.4744 \cdot 10^{-1}, \quad &q_2(x,0)= 1.8949 \cdot 10^{-1},\\
        \text{Artery $3$: } &a_3(x,0)=9.4744 \cdot 10^{-1}, \quad &q_3(x,0)= 1.8949 \cdot 10^{-1}.
    \end{aligned}
\end{equation}
In Fig. \ref{Fig:BF_1-2_Test1} we plot the solution, observing that the symmetry of the bifurcation results into a balance of the quantities involved in the second and third arteries.\\
Let us now consider a non-symmetric initial condition, 
given by
\begin{equation}
\label{BloodFlow_1-2_Test2}
    \begin{aligned}
        \text{Artery $1$: } &a_1(x,0)=1.8055, \quad &q_1(x,0)= 3.6110 \cdot 10^{-1},\\
        \text{Artery $2$: } &a_2(x,0)=1.13097, \quad &q_2(x,0)= 2.2619 \cdot 10^{-1},\\
        \text{Artery $3$: } &a_3(x,0)=6.3617 \cdot 10^{-1}, \quad &q_3(x,0)= 1.2723 \cdot 10^{-1}.
    \end{aligned}
\end{equation}
In Fig. \ref{Fig:BF_1-2_Test1} we plot the solution, where the non-symmetric initial condition leads to a non-symmetric evolution both of the sections $a_2$, $a_3$ and the discharges $q_2$, $q_3$. 

\begin{figure}[ht!]
    \begin{minipage}[h]{0.49\linewidth}
    \begin{center}
    \includegraphics[width=0.75\linewidth]{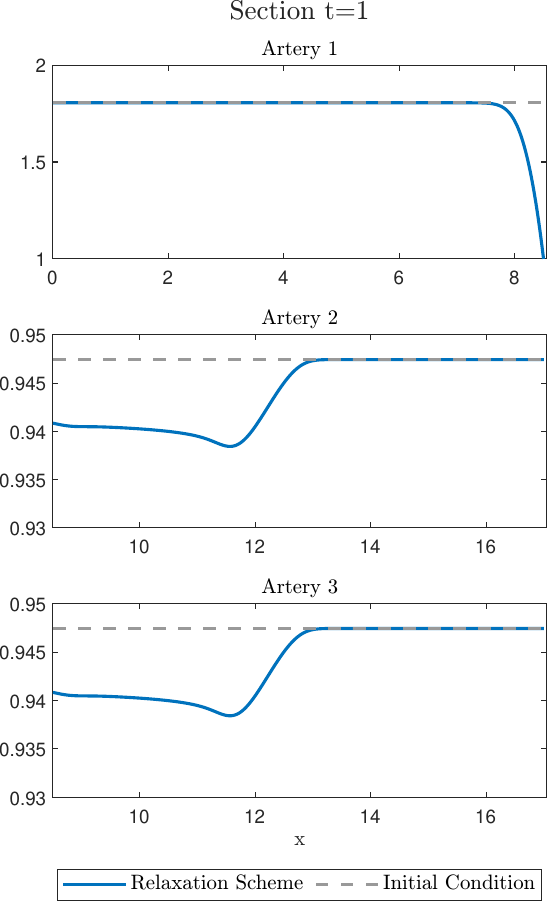} 
    \end{center}
    \end{minipage}
    \hfill
    \begin{minipage}[h]{0.49\linewidth}
    \begin{center}
    \includegraphics[width=0.75\linewidth]{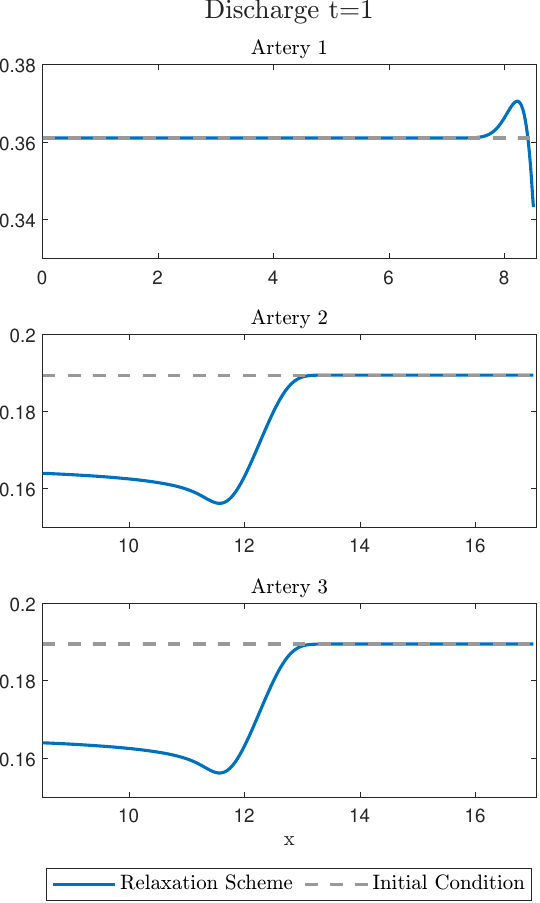} 
    \end{center} 
    \end{minipage}
    \caption{(\textit{One Incoming - Two Outgoing}) Initial condition \eqref{BloodFlow_1-2_Test1} with $x \in (0,17)$ and the junction in $x=8.5$. We show the numerical solution at time $t=1$.}
    \label{Fig:BF_1-2_Test1}
\end{figure}

\begin{figure}[ht!]
    \begin{minipage}[h]{0.49\linewidth}
    \begin{center}
    \includegraphics[width=0.75\linewidth]{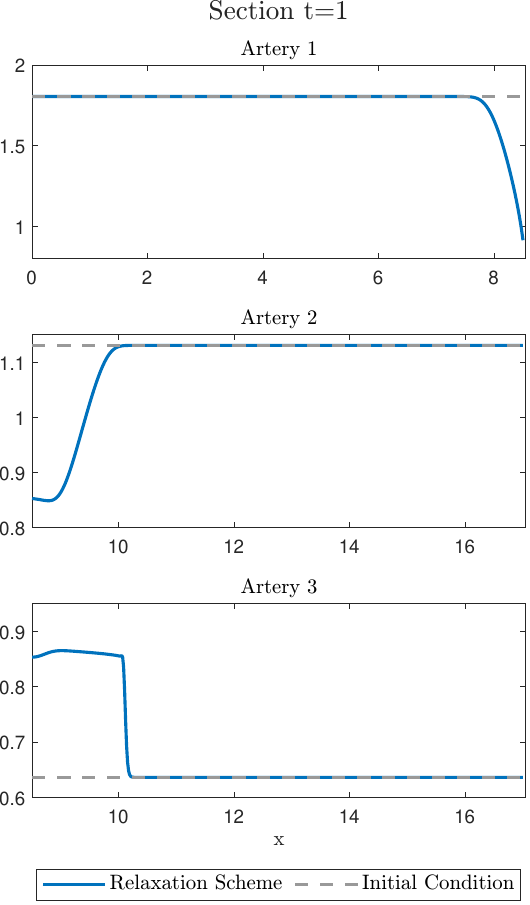} 
    \end{center}
    \end{minipage}
    \hfill
    \begin{minipage}[h]{0.49\linewidth}
    \begin{center}
    \includegraphics[width=0.75\linewidth]{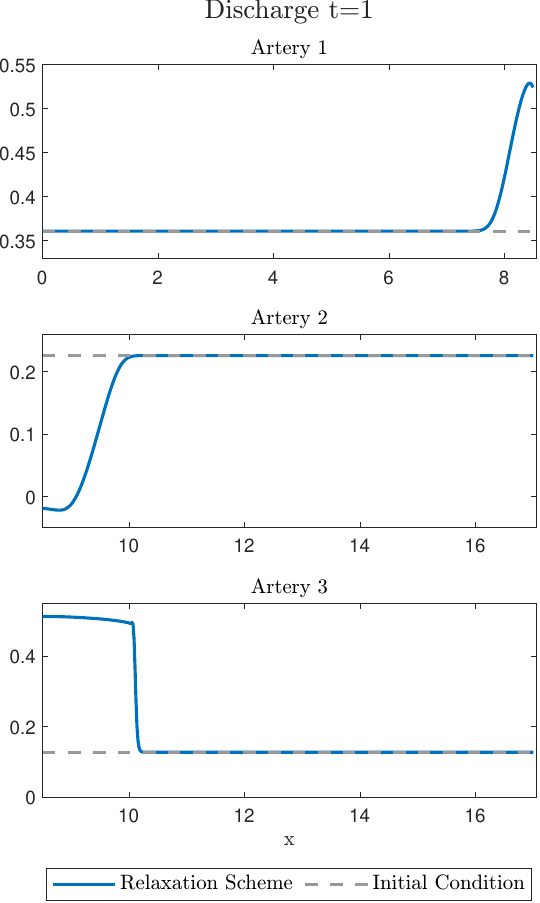} 
    \end{center} 
    \end{minipage}
    \caption{(\textit{One Incoming - Two Outgoing}) Initial condition \eqref{BloodFlow_1-2_Test2} with $x \in (0, 17)$ and the junction in $x=8.55$. We show the numerical solution at time $t=1$.}
    \label{Fig:BF_1-2_Test2}
\end{figure}

\FloatBarrier

\section{Conclusion}
\label{sec:Conclusion}
In this paper we proposed a new approach based on relaxation schemes to approximate hyperbolic conservation laws in networks. The most problematic aspect of networks is the approximation of solutions at the junction, which usually requires the use of approximate Riemann solver. The introduction of relaxation schemes allows to explicitly compute the admissible solutions at the junction, without increasing the computational cost. This strategy has been applied to shallow water equations and blood flow equations, proving efficacy in the treatment of both subcritical and supercritical conditions. Well-balanced techniques are considered to treat source terms, leading to steady-state preserving schemes.\\
Further perspectives include the application of this strategy to different types of equations, including networks of advection-diffusion equations \cite{tenna2025}. Moreover, including source terms in chemotaxis equations in networks with jump transmissions conditions \cite{BNR2025} represents a challenging research project, since jump discontinuities at the junction need proper treatment in the reconstruction procedure of well-balanced schemes.  

\section*{Acknowledgement}
The author received funding from the European Union's Horizon Europe research and innovation program under the Marie Sklodowska-Curie Doctoral Network DataHyKing (Grant No. 101072546). The author is member of GNCS-INdAM research group.

\appendix
\section{Solutions to the Riemann problem}
\label{Appendix:Solution_Riemann}
\subsubsection{The Riemann Problem for Shallow Waters on a Flat Bottom}
Let us consider the Riemann problem for the homogeneous system
\begin{equation}
    \begin{cases}
        \partial_t u + \partial_x F(u) = 0,\\[10pt]
        u(x,0)=\begin{cases}
            u_\ell \qquad &\text{ if } x \leq 0,\\
            u_r \qquad &\text{ if } x > 0,
        \end{cases}
    \end{cases}
\end{equation}
where the left and right states are given by
\begin{equation}
    u_\ell=\begin{pmatrix}
		h_\ell\\
		h_\ell\,v_\ell
	\end{pmatrix}, \qquad 
    u_r=\begin{pmatrix}
		h_r\\
		h_r\,v_r
	\end{pmatrix}
\end{equation}
The Riemann solution consists of two waves, each of which is a shock or rarefaction, related to different eigenvectors. In the fluvial regime, due to the subcritical condition \eqref{eq::chap4_subcritical_condition}, there will be one left wave with negative speed and one right wave with positive speed. The solution consists of the $\ell$-wave (corresponding to the left going wave) and the $r$-wave (corresponding to the right going wave), separated by an intermediate state $u^*=(h^*, q^*)$. To determine $h^*$ we need to impose some conditions which guarantee that the intermediate states are connected by a physically correct $\ell$- (or $r$-) wave. To this aim, we can define two functions $\Phi_\ell$ and $\Phi_r$ by
\begin{equation}
\label{Phi_L_Riemann}
    \Phi_\ell(h) = \begin{cases}
        v_\ell - 2 \left(\sqrt{gh} - \sqrt{gh_\ell} \right), \qquad &\text{if } h < h_\ell \quad \text{(rarefaction)}\\[10pt]
        v_\ell - (h-h_\ell)\sqrt{g \displaystyle \frac{h+h_\ell}{2\,h\,h_\ell}}, \qquad &\text{if } h > h_\ell \quad \text{(shock wave)},\\
    \end{cases}
\end{equation}
and
\begin{equation}
\label{Phi_R_Riemann}
    \Phi_r(h) = \begin{cases}
        v_r + 2 \left(\sqrt{gh} - \sqrt{gh_r} \right), \qquad &\text{if } h < h_r \quad \text{(rarefaction)}\\[10pt]
        v_r + (h-h_r)\sqrt{g \displaystyle \frac{h+h_r}{2\,h\,h_r}}, \qquad &\text{if } h > h_r \quad \text{(shock wave)}.\\
    \end{cases}
\end{equation}
The functions $\Phi_\ell(h)$ and $\Phi_r(h)$ return the value of $v$ such that $(h,hv)$ can be connected with the physically correct wave to $u_\ell$ and to $u_r$ respectively. The intermediate state $w^*$ is determined so that $\Phi_\ell(h^*)=\Phi_r(h^*)$. This means that the intermediate state is unique if in the $(h,v)$-plane the two curves $\Phi_\ell(h)$ and $\Phi_r(h)$ intersect in $(h^*, v^*)$. Since $\Phi_\ell$ is strictly decreasing and unbounded and $\Phi_r$ is strictly increasing and unbounded, this is equivalent to  
\begin{equation}
\label{inequality_riemann}
    v_\ell + 2 \sqrt{gh_\ell} = \Phi_\ell(0) \, \geq \, \Phi_r(0) = v_r - 2 \sqrt{gh_r}.
\end{equation}
In the region defined by \eqref{inequality_riemann}, the Riemann problem has a unique solution.
We will also consider the regime transitional curves, which corresponds to the critical case $|v| = \sqrt{g\,h}$, namely
\begin{equation}
    \mathcal{C}^+ = \sqrt{g\,h}, \qquad \mathcal{C}^- = -\sqrt{g\,h}.
\end{equation}
For a detailed analysis on the Lax-curves for the shallow water system we refer to \cite{brianipiccoli2018}.

\subsubsection{The Riemann Problem for the Arterial Blood Flow on a Flat Bottom}
Let us consider again the Riemann problem for the homogeneous system
\begin{equation}
    \begin{cases}
        \partial_t u + \partial_x F(u) = 0,\\[10pt]
        u(x,0)=\begin{cases}
            u_\ell \qquad &\text{ if } x \leq 0,\\
            u_r \qquad &\text{ if } x > 0,
        \end{cases}
    \end{cases}
\end{equation}
where now the left and right states are given by
\begin{equation}
    u_\ell=\begin{pmatrix}
		a_\ell\\
		a_\ell\,v_\ell
	\end{pmatrix}, \qquad 
    u_r=\begin{pmatrix}
		a_r\\
		a_r\,v_r
	\end{pmatrix}
\end{equation}
Also in this case the Riemann solution consists of two waves, each of which is a shock or rarefaction, related to different eigenvectors \eqref{eigenvectors_blood}. We explicit give the solution to the Riemann problem in the case $\nu=0$ and $\alpha=1$. Analogously to the shallow water case, we can define two functions $\Psi_\ell$ and $\Psi_r$ as
\begin{equation}
\label{Psi_L_Riemann}
    \Psi_\ell(a) = \begin{cases}
        v_\ell - 4 \sqrt{ \displaystyle \frac{\beta}{2\,\rho}} \Bigg( \displaystyle \sqrt[4]{\frac{a}{a_0}} - \sqrt[4]{\frac{a_\ell}{a_0}} \Bigg), \qquad &\text{if } a < a_\ell \quad \text{(rarefaction)}\\[13pt]
        v_\ell - \sqrt{\displaystyle \frac{(a-a_\ell)}{a\,a_\ell} \, \frac{\pi(a)+\pi(a_\ell)}{\rho}}, \qquad &\text{if } a > a_\ell \quad \text{(shock wave)},\\
    \end{cases}
\end{equation}
and
\begin{equation}
\label{Psi_R_Riemann}
    \Psi_r(a) = \begin{cases}
       v_r + 4 \sqrt{ \displaystyle \frac{\beta}{2\,\rho}} \Bigg( \displaystyle \sqrt[4]{\frac{a}{a_0}} - \sqrt[4]{\frac{a_r}{a_0}} \Bigg), \qquad &\text{if } a < a_r \quad \text{(rarefaction)}\\[13pt]
        v_r + \sqrt{\displaystyle \frac{(a-a_r)}{a\,a_r} \, \frac{\pi(a)+\pi(a_r)}{\rho}}, \qquad &\text{if } a > a_r \quad \text{(shock wave)},\\
    \end{cases}
\end{equation}
The functions $\Psi_\ell(h)$ and $\Psi_r(h)$ return the value of $v$ such that $(a,av)$ can be connected with the physically correct wave to $u_\ell$ and to $u_r$ respectively.

\bibliography{References}

@Book{bouchut2004,
 Author = {Bouchut, Fran{\c{c}}ois},
 Title = {Nonlinear stability of finite volume methods for hyperbolic conservation laws and well-balanced schemes for sources.},
 FSeries = {Frontiers in Mathematics},
 Series = {Front. Math.},
 ISSN = {1660-8046},
 ISBN = {3-7643-6665-6},
 Year = {2004},
 Publisher = {Basel: Birkh{\"a}user},
 Language = {English},
 zbMATH = {2174321},
 Zbl = {1086.65091}
}

@Article{brianipiccoli2018,
 Author = {Briani, Maya and Piccoli, Benedetto},
 Title = {Fluvial to torrential phase transition in open canals},
 FJournal = {Networks and Heterogeneous Media},
 Journal = {Netw. Heterog. Media},
 ISSN = {1556-1801},
 Volume = {13},
 Number = {4},
 Pages = {663--690},
 Year = {2018},
 Language = {English},
 DOI = {10.3934/nhm.2018030},
 zbMATH = {7058969},
 Zbl = {1420.35224}
}

@Article{brianipiccoli2016,
 Author = {Briani, Maya and Piccoli, Benedetto and Qiu, Jing-Mei},
 Title = {Notes on {RKDG} methods for shallow-water equations in canal networks},
 FJournal = {Journal of Scientific Computing},
 Journal = {J. Sci. Comput.},
 ISSN = {0885-7474},
 Volume = {68},
 Number = {3},
 Pages = {1101--1123},
 Year = {2016},
 Language = {English},
 DOI = {10.1007/s10915-016-0172-2},
 zbMATH = {6645168},
 Zbl = {1437.76020}
}

@article{audusse2004,
 author = {Audusse, Emmanuel and Bouchut, Fran{\c{c}}ois and Bristeau, Marie-Odile and Klein, Rupert and Perthame, Beno{\^{\i}}t},
 title = {A fast and stable well-balanced scheme with hydrostatic reconstruction for shallow water flows},
 fjournal = {SIAM Journal on Scientific Computing},
 journal = {SIAM J. Sci. Comput.},
 issn = {1064-8275},
 volume = {25},
 number = {6},
 pages = {2050--2065},
 year = {2004},
 language = {English},
 doi = {10.1137/S1064827503431090},
 zbMATH = {2138724},
 Zbl = {1133.65308}
}

@article{noelle2006,
 author = {Noelle, Sebastian and Pankratz, Normann and Puppo, Gabriella and Natvig, Jostein R.},
 title = {Well-balanced finite volume schemes of arbitrary order of accuracy for shallow water flows},
 fjournal = {Journal of Computational Physics},
 journal = {J. Comput. Phys.},
 issn = {0021-9991},
 volume = {213},
 number = {2},
 pages = {474--499},
 year = {2006},
 language = {English},
 doi = {10.1016/j.jcp.2005.08.019},
 zbMATH = {5020831},
 Zbl = {1088.76037}
}

@article{aregba2000discrete,
 author = {Aregba-Driollet, Denise and Natalini, Roberto},
 title = {Discrete kinetic schemes for multidimensional systems of conservation laws},
 fjournal = {SIAM Journal on Numerical Analysis},
 journal = {SIAM J. Numer. Anal.},
 issn = {0036-1429},
 volume = {37},
 number = {6},
 pages = {1973--2004},
 year = {2000},
 language = {English},
 doi = {10.1137/S0036142998343075},
 zbMATH = {1519661},
 Zbl = {0964.65096}
}

@article{bressanpiccoli2014,
 author = {Bressan, Alberto and {\v{C}}ani{\'c}, Sun{\v{c}}ica and Garavello, Mauro and Herty, Michael and Piccoli, Benedetto},
 title = {Flows on networks: recent results and perspectives},
 fjournal = {EMS Surveys in Mathematical Sciences},
 journal = {EMS Surv. Math. Sci.},
 issn = {2308-2151},
 volume = {1},
 number = {1},
 pages = {47--111},
 year = {2014},
 language = {English},
 doi = {10.4171/EMSS/2},
 zbMATH = {6316076},
 Zbl = {1301.35193}
}

@book{garavellopiccoli2006,
 author = {Garavello, Mauro and Piccoli, Benedetto},
 title = {Traffic flow on networks},
 fseries = {AIMS Series on Applied Mathematics},
 series = {AIMS Ser. Appl. Math.},
 volume = {1},
 isbn = {978-1-60133-000-0},
 year = {2006},
 publisher = {Springfield, MO: American Institute of Mathematical Sciences (AIMS)},
 language = {English},
 zbMATH = {5130020},
 Zbl = {1136.90012}
}

@article{canickim2003,
 author = {{\v{C}}ani{\'c}, Sun{\v{c}}ica and Kim, Eun Heui},
 title = {Mathematical analysis of the quasilinear effects in a hyperbolic model blood flow through compliant axi-symmetric vessels},
 fjournal = {Mathematical Methods in the Applied Sciences},
 journal = {Math. Methods Appl. Sci.},
 issn = {0170-4214},
 volume = {26},
 number = {14},
 pages = {1161--1186},
 year = {2003},
 language = {English},
 doi = {10.1002/mma.407},
 zbMATH = {1997831},
 Zbl = {1141.76484}
}

@article{bastin2009open,
  title={Open problems and research perspectives for irrigation channels},
  author={Bastin, Georges and Bayen, Alexandre M and D’Apice, CIRO and Litrico, Xavier and Piccoli, Benedetto and others},
  journal={Netw. Heterog. Media},
  volume={4},
  number={2},
  year={2009}
}

@article{perthame2001,
 author = {Perthame, B. and Simeoni, C.},
 title = {A kinetic scheme for the {Saint}-{Venant} system with a source term},
 fjournal = {Calcolo},
 journal = {Calcolo},
 issn = {0008-0624},
 volume = {38},
 number = {4},
 pages = {201--231},
 year = {2001},
 language = {English},
 doi = {10.1007/s10092-001-8181-3},
 zbMATH = {1846809},
 Zbl = {1008.65066}
}

@article{ranocha2017,
 author = {Ranocha, Hendrik},
 title = {Shallow water equations: split-form, entropy stable, well-balanced, and positivity preserving numerical methods},
 fjournal = {GEM - International Journal on Geomathematics},
 journal = {GEM. Int. J. Geomath.},
 issn = {1869-2672},
 volume = {8},
 number = {1},
 pages = {85--133},
 year = {2017},
 language = {English},
 doi = {10.1007/s13137-016-0089-9},
 zbMATH = {6853862},
 Zbl = {1432.65156}
}

@article{fjordholm2011,
 author = {Fjordholm, Ulrik S. and Mishra, Siddhartha and Tadmor, Eitan},
 title = {Well-balanced and energy stable schemes for the shallow water equations with discontinuous topography},
 fjournal = {Journal of Computational Physics},
 journal = {J. Comput. Phys.},
 issn = {0021-9991},
 volume = {230},
 number = {14},
 pages = {5587--5609},
 year = {2011},
 language = {English},
 doi = {10.1016/j.jcp.2011.03.042},
 zbMATH = {5920301},
 Zbl = {1452.35149}
}

@article{dansac2016,
 author = {Michel-Dansac, Victor and Berthon, Christophe and Clain, St{\'e}phane and Foucher, Fran{\c{c}}oise},
 title = {A well-balanced scheme for the shallow-water equations with topography},
 fjournal = {Computers \& Mathematics with Applications},
 journal = {Comput. Math. Appl.},
 issn = {0898-1221},
 volume = {72},
 number = {3},
 pages = {568--593},
 year = {2016},
 language = {English},
 doi = {10.1016/j.camwa.2016.05.015},
 zbMATH = {6701755},
 Zbl = {1359.76206}
}

@article{castro2006,
 author = {Castro, Manuel and Gallardo, Jos{\'e} M. and Par{\'e}s, Carlos},
 title = {High order finite volume schemes based on reconstruction of states for solving hyperbolic systems with nonconservative products. {Applications} to shallow-water systems},
 fjournal = {Mathematics of Computation},
 journal = {Math. Comput.},
 issn = {0025-5718},
 volume = {75},
 number = {255},
 pages = {1103--1134},
 year = {2006},
 language = {English},
 doi = {10.1090/S0025-5718-06-01851-5},
 zbMATH = {5028528},
 Zbl = {1096.65082}
}

@article{delestre2013,
 author = {Delestre, O. and Lagr{\'e}e, P.-Y.},
 title = {A `well-balanced' finite volume scheme for blood flow simulation},
 fjournal = {International Journal for Numerical Methods in Fluids},
 journal = {Int. J. Numer. Methods Fluids},
 issn = {0271-2091},
 volume = {72},
 number = {2},
 pages = {177--205},
 year = {2013},
 language = {English},
 doi = {10.1002/fld.3736},
 zbMATH = {7317843},
 Zbl = {1455.76215}
}

@book{formaggiaquarteroni2009,
 editor = {Formaggia, Luca and Quarteroni, Alfio and Veneziani, Alessandro},
 title = {Cardiovascular mathematics. {Modeling} and simulation of the circulatory system},
 fseries = {MS\&A. Modeling, Simulation and Applications},
 series = {MS\&A, Model. Simul. Appl.},
 issn = {2037-5255},
 volume = {1},
 isbn = {978-88-470-1151-9; 978-88-470-1152-6},
 year = {2009},
 publisher = {Milano: Springer},
 language = {English},
 doi = {10.1007/978-88-470-1152-6},
 zbMATH = {5368508},
 Zbl = {1300.92005}
}

@article{bertagliapareschi2023,
 author = {Bertaglia, Giulia and Pareschi, Lorenzo},
 title = {Multiscale constitutive framework of one-dimensional blood flow modeling: asymptotic limits and numerical methods},
 fjournal = {Multiscale Modeling \& Simulation},
 journal = {Multiscale Model. Simul.},
 issn = {1540-3459},
 volume = {21},
 number = {3},
 pages = {1237--1267},
 year = {2023},
 language = {English},
 doi = {10.1137/23M1554230},
 zbMATH = {7752613},
 Zbl = {1525.76118}
}

@article{brettinatalini2018,
 author = {Bretti, Gabriella and Natalini, Roberto},
 title = {Numerical approximation of nonhomogeneous boundary conditions on networks for a hyperbolic system of chemotaxis modeling the {Physarum} dynamics},
 fjournal = {Journal of Computational Methods in Sciences and Engineering},
 journal = {J. Comput. Methods Sci. Eng.},
 issn = {1472-7978},
 volume = {18},
 number = {1},
 pages = {85--115},
 year = {2018},
 language = {English},
 doi = {10.3233/JCM-170773},
 zbMATH = {6941756},
 Zbl = {1398.92032}
}

@article{brettinataliniribot2014,
 author = {Bretti, G. and Natalini, R. and Ribot, M.},
 title = {A hyperbolic model of chemotaxis on a network: a numerical study},
 fjournal = {European Series in Applied and Industrial Mathematics (ESAIM): Mathematical Modelling and Numerical Analysis},
 journal = {ESAIM, Math. Model. Numer. Anal.},
 issn = {0764-583X},
 volume = {48},
 number = {1},
 pages = {231--258},
 year = {2014},
 language = {English},
 doi = {10.1051/m2an/2013098},
 zbMATH = {6291616},
 Zbl = {1285.92004}
}

@article{brettinatalinipiccoli2007,
 author = {Bretti, Gabriella and Natalini, Roberto and Piccoli, Benedetto},
 title = {A fluid-dynamic traffic model on road networks},
 fjournal = {Archives of Computational Methods in Engineering},
 journal = {Arch. Comput. Methods Eng.},
 issn = {1134-3060},
 volume = {14},
 number = {2},
 pages = {139--172},
 year = {2007},
 language = {English},
 doi = {10.1007/s11831-007-9004-8},
 zbMATH = {5223533},
 Zbl = {1127.76004}
}

@incollection{garavellopiccoli2009,
 author = {Garavello, Mauro and Piccoli, Benedetto},
 title = {Riemann solvers for conservation laws at a node},
 booktitle = {Hyperbolic problems. Theory, numerics and applications. Contributed talks. Proceedings of the 12th international conference on hyperbolic problems, June 9--13, 2008},
 isbn = {978-0-8218-4730-5; 978-0-8218-4728-2},
 pages = {595--604},
 year = {2009},
 publisher = {Providence, RI: American Mathematical Society (AMS)},
 language = {English},
 zbMATH = {5674636},
 Zbl = {1185.65147}
}

@article{colombogoatinpiccoli2010,
 author = {Colombo, Rinaldo M. and Goatin, Paola and Piccoli, Benedetto},
 title = {Road networks with phase transitions},
 fjournal = {Journal of Hyperbolic Differential Equations},
 journal = {J. Hyperbolic Differ. Equ.},
 issn = {0219-8916},
 volume = {7},
 number = {1},
 pages = {85--106},
 year = {2010},
 language = {English},
 doi = {10.1142/S0219891610002025},
 zbMATH = {5709668},
 Zbl = {1189.35176}
}

@misc{BNR2025,
 author = {Briani, Maya and Natalini, Roberto and Ribot, Magali},
 title = {A relaxation scheme for the equations of isentropic gas dynamics on a network with jump transmission conditions},
 year = {2025},
 howpublished = {Preprint, {arXiv}:2507.05779 [math.{NA}] (2025)},
 url = {https://arxiv.org/abs/2507.05779},
 arXiv = {arXiv:2507.05779}
}

@Article{aregba1996,
 Author = {Aregba-Driollet, Denise and Natalini, Roberto},
 Title = {Convergence of relaxation schemes for conservation laws},
 FJournal = {Applicable Analysis},
 Journal = {Appl. Anal.},
 ISSN = {0003-6811},
 Volume = {61},
 Number = {1-2},
 Pages = {163--190},
 Year = {1996},
 Language = {English},
 DOI = {10.1080/00036819608840453},
 zbMATH = {1022094},
 Zbl = {0872.65086}
}

@Article{lafittemelis2017,
 Author = {Lafitte, Pauline and Melis, Ward and Samaey, Giovanni},
 Title = {A high-order relaxation method with projective integration for solving nonlinear systems of hyperbolic conservation laws},
 FJournal = {Journal of Computational Physics},
 Journal = {J. Comput. Phys.},
 ISSN = {0021-9991},
 Volume = {340},
 Pages = {1--25},
 Year = {2017},
 Language = {English},
 DOI = {10.1016/j.jcp.2017.03.027},
 URL = {lirias.kuleuven.be/handle/123456789/578632},
 zbMATH = {6818910},
 Zbl = {1380.65161}
}

@misc{tenna2025,
 author = {Tenna, Tommaso},
 title = {Projective integration schemes for nonlinear degenerate parabolic systems},
 year = {2025},
 howpublished = {Preprint, {arXiv}:2503.05017 [math.{NA}] (2025)},
 url = {https://arxiv.org/abs/2503.05017},
 arXiv = {arXiv:2503.05017}
}

@book{boscarino2025,
 author = {Boscarino, S. and Pareschi, L. and Russo, G.},
 title = {Implicit-explicit methods for evolutionary partial differential equations},
 fseries = {Mathematical Modeling and Computation},
 series = {Math. Model. Comput.},
 volume = {24},
 isbn = {978-1-61197-819-3; 978-1-61197-820-9},
 year = {2025},
 publisher = {Philadelphia, PA: Society for Industrial {and} Applied Mathematics (SIAM)},
 language = {English},
 doi = {10.1137/1.9781611978209},
 zbMATH = {7927359}
}

@article{desveaux2022,
 author = {Desveaux, Vivien and Masset, Alice},
 title = {A fully well-balanced scheme for shallow water equations with {Coriolis} force},
 fjournal = {Communications in Mathematical Sciences},
 journal = {Commun. Math. Sci.},
 issn = {1539-6746},
 volume = {20},
 number = {7},
 pages = {1875--1900},
 year = {2022},
 language = {English},
 doi = {10.4310/CMS.2022.v20.n7.a4},
 zbMATH = {7632190},
 Zbl = {1499.65458}
}

@article{audusse2021,
 author = {Audusse, E. and Dubos, V. and Duran, A. and Gaveau, N. and Nasseri, Y. and Penel, Y.},
 title = {Numerical approximation of the shallow water equations with {Coriolis} source term},
 fjournal = {European Series in Applied and Industrial Mathematics (ESAIM): Proceedings and Surveys},
 journal = {ESAIM, Proc. Surv.},
 issn = {2267-3059},
 volume = {70},
 pages = {31--44},
 year = {2021},
 language = {English},
 doi = {10.1051/proc/202107003},
 zbMATH = {7752051},
 Zbl = {1525.76065}
}

@article{lizhao2011,
 author = {Li, Tong and Zhao, Kun},
 title = {Global existence and long-time behavior of entropy weak solutions to a quasilinear hyperbolic blood flow model},
 fjournal = {Networks and Heterogeneous Media},
 journal = {Netw. Heterog. Media},
 issn = {1556-1801},
 volume = {6},
 number = {4},
 pages = {625--646},
 year = {2011},
 language = {English},
 doi = {10.3934/nhm.2011.6.625},
 zbMATH = {6147610},
 Zbl = {1262.35156}
}

@article{leveque1998,
 author = {LeVeque, Randall J.},
 title = {Balancing source terms and flux gradients in high-resolution {Godunov} methods: {The} quasi-steady wave-propagation algorithm},
 fjournal = {Journal of Computational Physics},
 journal = {J. Comput. Phys.},
 issn = {0021-9991},
 volume = {146},
 number = {1},
 pages = {346--365},
 year = {1998},
 language = {English},
 doi = {10.1006/jcph.1998.6058},
 url = {semanticscholar.org/paper/388011d5ea8825e6e54f8d6bb3018e62752aba79},
 zbMATH = {1240722},
 Zbl = {0931.76059}
}

@article{greenberg1996,
 author = {Greenberg, James M. and Le Roux, Alain-Yves},
 title = {A well-balanced scheme for the numerical processing of source terms in hyperbolic equations},
 fjournal = {SIAM Journal on Numerical Analysis},
 journal = {SIAM J. Numer. Anal.},
 issn = {0036-1429},
 volume = {33},
 number = {1},
 pages = {1--16},
 year = {1996},
 language = {English},
 doi = {10.1137/0733001},
 zbMATH = {869875},
 Zbl = {0876.65064}
}

@article{gosse2001,
 author = {Gosse, Laurent},
 title = {A well-balanced scheme using non-conservative products designed for hyperbolic systems of conservation laws with source terms},
 fjournal = {M\(^3\)AS. Mathematical Models \& Methods in Applied Sciences},
 journal = {Math. Models Methods Appl. Sci.},
 issn = {0218-2025},
 volume = {11},
 number = {2},
 pages = {339--365},
 year = {2001},
 language = {English},
 doi = {10.1142/S021820250100088X},
 zbMATH = {1882777},
 Zbl = {1018.65108}
}

@article{gossetoscani2003,
 author = {Gosse, Laurent and Toscani, Giuseppe},
 title = {Space localization and well-balanced schemes for discrete kinetic models in diffusive regimes},
 fjournal = {SIAM Journal on Numerical Analysis},
 journal = {SIAM J. Numer. Anal.},
 issn = {0036-1429},
 volume = {41},
 number = {2},
 pages = {641--658},
 year = {2003},
 language = {English},
 doi = {10.1137/S0036142901399392},
 zbMATH = {2027737},
 Zbl = {1130.82340}
}

@article{gosse2000,
 author = {Gosse, L.},
 title = {A well-balanced flux-vector splitting scheme designed for hyperbolic systems of conservation laws with source terms},
 fjournal = {Computers \& Mathematics with Applications},
 journal = {Comput. Math. Appl.},
 issn = {0898-1221},
 volume = {39},
 number = {9-10},
 pages = {135--159},
 year = {2000},
 language = {English},
 doi = {10.1016/S0898-1221(00)00093-6},
 zbMATH = {1471226},
 Zbl = {0963.65090}
}

@book{pedley1980, 
place={Cambridge}, 
series={Cambridge Monographs on Mechanics}, 
title={The Fluid Mechanics of Large Blood Vessels},
publisher={Cambridge University Press},
author={Pedley, T. J.},
year={1980},
collection={Cambridge Monographs on Mechanics}
}

@article{boileau2015,
author = {Boileau, Etienne and Nithiarasu, Perumal and Blanco, Pablo J. and Müller, Lucas O. and Fossan, Fredrik Eikeland and Hellevik, Leif Rune and Donders, Wouter P. and Huberts, Wouter and Willemet, Marie and Alastruey, Jordi},
title = {A benchmark study of numerical schemes for one-dimensional arterial blood flow modelling},
journal = {International Journal for Numerical Methods in Biomedical Engineering},
volume = {31},
number = {10},
doi = {https://doi.org/10.1002/cnm.2732},
url = {https://onlinelibrary.wiley.com/doi/abs/10.1002/cnm.2732},
eprint = {https://onlinelibrary.wiley.com/doi/pdf/10.1002/cnm.2732},
year = {2015}
}

@article{olufsen2000numerical,
  title={Numerical simulation and experimental validation of blood flow in arteries with structured-tree outflow conditions},
  author={Olufsen, Mette S and Peskin, Charles S and Kim, Won Yong and Pedersen, Erik M and Nadim, Ali and Larsen, Jesper},
  journal={Annals of biomedical engineering},
  volume={28},
  number={11},
  pages={1281--1299},
  year={2000},
  publisher={Springer}
}

@article{jinxin1995,
 author = {Jin, Shi and Xin, Zhouping},
 title = {The relaxation schemes for systems of conservation laws in arbitrary space dimensions},
 fjournal = {Communications on Pure and Applied Mathematics},
 journal = {Commun. Pure Appl. Math.},
 issn = {0010-3640},
 volume = {48},
 number = {3},
 pages = {235--276},
 year = {1995},
 language = {English},
 doi = {10.1002/cpa.3160480303},
 zbMATH = {769165},
 Zbl = {0826.65078}
}

@article{pareschirusso2005,
 author = {Pareschi, Lorenzo and Russo, Giovanni},
 title = {Implicit-explicit {Runge}-{Kutta} schemes and applications to hyperbolic systems with relaxation},
 fjournal = {Journal of Scientific Computing},
 journal = {J. Sci. Comput.},
 issn = {0885-7474},
 volume = {25},
 number = {1-2},
 pages = {129--155},
 year = {2005},
 language = {English},
 doi = {10.1007/BF02728986},
 zbMATH = {5839668},
 Zbl = {1203.65111}
}

@article{torloabgrall2020,
 author = {Abgrall, R{\'e}mi and Torlo, Davide},
 title = {High order asymptotic preserving deferred correction implicit-explicit schemes for kinetic models},
 fjournal = {SIAM Journal on Scientific Computing},
 journal = {SIAM J. Sci. Comput.},
 issn = {1064-8275},
 volume = {42},
 number = {3},
 pages = {b816--b845},
 year = {2020},
 language = {English},
 doi = {10.1137/19M128973X},
 zbMATH = {7226088},
 Zbl = {1473.65183}
}

@article{pupposemplice2016,
 author = {Puppo, Gabriella and Semplice, Matteo},
 title = {Well-balanced high order {1D} schemes on non-uniform grids and entropy residuals},
 fjournal = {Journal of Scientific Computing},
 journal = {J. Sci. Comput.},
 issn = {0885-7474},
 volume = {66},
 number = {3},
 pages = {1052--1076},
 year = {2016},
 language = {English},
 doi = {10.1007/s10915-015-0056-x},
 zbMATH = {6760783},
 Zbl = {1371.65093}
}

@article{colombo2008,
 author = {Colombo, R. M. and Herty, M. and Sachers, V.},
 title = {On {{\(2\times2\)}} conservation laws at a junction},
 fjournal = {SIAM Journal on Mathematical Analysis},
 journal = {SIAM J. Math. Anal.},
 issn = {0036-1410},
 volume = {40},
 number = {2},
 pages = {605--622},
 year = {2008},
 language = {English},
 doi = {10.1137/070690298},
 zbMATH = {5534315},
 Zbl = {1171.35430}
}

@article{gugat2004,
 author = {Gugat, Martin and Leugering, G{\"u}nter and Schmidt, E. J. P. Georg},
 title = {Global controllability between steady supercritical flows in channel networks},
 fjournal = {Mathematical Methods in the Applied Sciences},
 journal = {Math. Methods Appl. Sci.},
 issn = {0170-4214},
 volume = {27},
 number = {7},
 pages = {781--802},
 year = {2004},
 language = {English},
 doi = {10.1002/mma.471},
 zbMATH = {2085662},
 Zbl = {1047.93028}
}

@article{muller2013,
 author = {M{\"u}ller, Lucas O. and Par{\'e}s, Carlos and Toro, Eleuterio F.},
 title = {Well-balanced high-order numerical schemes for one-dimensional blood flow in vessels with varying mechanical properties},
 fjournal = {Journal of Computational Physics},
 journal = {J. Comput. Phys.},
 issn = {0021-9991},
 volume = {242},
 pages = {53--85},
 year = {2013},
 language = {English},
 doi = {10.1016/j.jcp.2013.01.050},
 zbMATH = {6334650},
 Zbl = {1323.92066}
}

@article{lucca2025,
 author = {Lucca, A. and M{\"u}ller, L. O. and Fraccarollo, L. and Toro, E. F. and Dumbser, M.},
 title = {On simple well-balanced semi-implicit and explicit numerical methods for blood flow in networks of elastic vessels with applications to {FFR} prediction},
 fjournal = {Journal of Computational Physics},
 journal = {J. Comput. Phys.},
 issn = {0021-9991},
 volume = {538},
 pages = {34},
 note = {Id/No 114188},
 year = {2025},
 language = {English},
 doi = {10.1016/j.jcp.2025.114188},
 zbMATH = {8088189}
}

@article{muller2015,
 author = {M{\"u}ller, Lucas O. and Blanco, Pablo J.},
 title = {A high order approximation of hyperbolic conservation laws in networks: application to one-dimensional blood flow},
 fjournal = {Journal of Computational Physics},
 journal = {J. Comput. Phys.},
 issn = {0021-9991},
 volume = {300},
 pages = {423--437},
 year = {2015},
 language = {English},
 doi = {10.1016/j.jcp.2015.07.056},
 url = {hdl.handle.net/11572/230377},
 zbMATH = {6660762},
 Zbl = {1349.76945}
}

@article{borsche2014,
	author = {Borsche, R. and G{\"o}ttlich, S. and Klar, A. and Schillen, P.},
	title = {The scalar {Keller}-{Segel} model on networks},
	fjournal = {M\(^3\)AS. Mathematical Models \& Methods in Applied Sciences},
	journal = {Math. Models Methods Appl. Sci.},
	issn = {0218-2025},
	volume = {24},
	number = {2},
	pages = {221--247},
	year = {2014},
	language = {English},
	doi = {10.1142/S0218202513400071},
	zbMATH = {6250057},
	Zbl = {1321.92042}
}

@article{borsche2014_traffic,
	author = {Borsche, R. and Klar, A. and K{\"u}hn, S. and Meurer, A.},
	title = {Coupling traffic flow networks to pedestrian motion},
	fjournal = {M\(^3\)AS. Mathematical Models \& Methods in Applied Sciences},
	journal = {Math. Models Methods Appl. Sci.},
	issn = {0218-2025},
	volume = {24},
	number = {2},
	pages = {359--380},
	year = {2014},
	language = {English},
	doi = {10.1142/S0218202513400113},
	url = {kluedo.ub.rptu.de/frontdoor/index/index/docId/3529},
	zbMATH = {6250061},
	Zbl = {1279.90035}
}
\bibliographystyle{acm}
\end{document}